\documentclass[12pt,a4paper]{article}
\usepackage{amsfonts}
\usepackage{subfigure}
\usepackage{comment}
\usepackage{color}
\usepackage{mathrsfs}
\usepackage{stmaryrd}
\usepackage{amsmath}
\usepackage{amssymb}
\usepackage{bbm}
\usepackage{graphicx}
\usepackage{algorithm}
\usepackage{algorithmic}
\usepackage{tikz}
\usetikzlibrary{matrix}
\usetikzlibrary{arrows}
\usepackage{pgfplots}
\pgfplotsset{compat=1.18}
\usepackage{enumerate}
\usepackage{theorem}
\usepackage{ulem}

\newtheorem{theorem}{Theorem}[section]
\newtheorem{proposition}{Proposition}[section]
\newtheorem{corollary}{Corollary}[section]
\newtheorem{conjecture}{Conjecture}[section]
\newtheorem{lemma}{Lemma}[section]
\newtheorem{definition}{Definition}[section]
\newtheorem{remark}{Remark}[section]
\newtheorem{example}{Example}[section]
\newtheorem{proof}{Proof}[section]

\usepackage{etoolbox}
\usepackage{url}
\DeclareRobustCommand{\rev}[1]{{\color{blue}#1}}
\AtBeginEnvironment{thebibliography}{%
  \let\originalbibitem\bibitem
  \renewcommand{\bibitem}[2][]{%
    \color{black}%
    \ifstrequal{#2}{sm1}{\color{blue}}{}%
    \ifstrempty{#1}{\originalbibitem{#2}}{\originalbibitem[#1]{#2}}%
  }%
}
\AtEndEnvironment{thebibliography}{\color{black}}
\DeclareMathOperator{\rank}{rank}
\DeclareMathOperator{\lev}{lev}

\begin{document}

\title{\Large\bf Dimension of Bi-degree $(d,d)$ Spline Spaces with the Highest Order of Smoothness over Hierarchical T-Meshes}
\author{Bingru~Huang \quad and \quad Falai Chen\thanks{Corresponding author.\\E-mail address: chenfl@ustc.edu.cn}\\\small School of Mathematical Sciences\\ \small University of Science and Technology of China\\ \small Hefei, Anhui 230026, P.R. of China}
\date{}
\maketitle
\begin{center}
\begin{minipage}{160mm}
{\bf Abstract.}
In this paper, we study the dimension of the spline space of bi-degree $(d,d)$ with the highest order of smoothness over a hierarchical T-mesh $\mathscr T$ using the smoothing cofactor-conformality method. Firstly, we obtain a dimensional formula for the conformality vector space over a tensor product T-connected component. Then, we prove that the dimension of the conformality vector space over a T-connected component of a hierarchical T-mesh under the tensor product subdivision can be calculated in a recursive manner. Combining these two aspects, we obtain a dimensional formula for the bi-degree $(d,d)$ spline space with the highest order of smoothness over a hierarchical T-mesh $\mathscr T$ with mild assumptions. Additionally, we provide a strategy to modify an arbitrary hierarchical T-mesh such that the dimension of the bi-degree $(d,d)$ spline space is stable over the modified hierarchical T-mesh. Finally, we prove that the dimension of the spline space over such a hierarchical T-mesh is the same as that of a lower-degree spline space over its CVR graph. Thus, the proposed solution can pave the way for the subsequent construction of basis functions for the spline space over such a hierarchical T-mesh.
\end{minipage}
\end{center}
\section{Introduction}
Locally refinable splines have gained much attention in the past two decades due to \rev{their} promising applications in iso-geometric analysis, and several types of such splines have been developed, such as hierarchical B-splines~\cite{HB1,HB4,HB3,HB2}, T-splines~\cite{ts2,ts1}, analysis-suitable T-splines~\cite{ast}, polynomial splines over hierarchical T-meshes~\cite{deng2008,surface} (PHT-splines in short) and locally refined splines~\cite{lr} (LR-splines in short). A comprehensive review of spline types is provided in~\cite{li2016survey}. One fundamental challenge in understanding these locally refinable splines is understanding the piecewise polynomial spline space over T-meshes, called polynomial splines over T-meshes. In this paper, we primarily study the dimensions of the polynomial splines over T-meshes.

The notion of polynomial splines over T-meshes $S_{\mathbf d}^{\mathbf r}(\mathscr T)$ was first proposed in~\cite{dim2006} as a bi-degree $\mathbf d=(d_1,d_2)$ piecewise polynomial spline space over a T-mesh $\mathscr T$ with a pair of smoothness orders $\mathbf r=(r_1,r_2)$ in the horizontal and vertical directions, respectively. This type of splines is defined directly from the viewpoint of spline spaces. Thus, two fundamental theoretical problems arise with polynomial splines over T-meshes: dimension calculation and basis construction. Over the past two decades, a lot of research has investigated these two basic issues.

Regarding dimension calculation, Deng et al.~\cite{dim2006} provided a dimension formula for $2\mathbf r+1\le\mathbf d$ using the B-net method. Later on, a dimension formula was provided in~\cite{dim20061} using the smoothness cofactor method and the same result as~\cite{dim2006} was obtained. However, one term in the general formula was given in a non-explicit form. Subsequently, the minimal determining set method was proposed in~\cite{MDS,schumaker2012}, and the implementation of the method showed the same result as in~\cite{dim2006}. Mourrain~\cite{dim2014} presented a general dimension formula and provided the upper and lower bounds of the dimension formula by the homology method. Similar to the dimension formula presented in~\cite{dim20061}, the general dimension formula cannot be used to calculate the dimensions directly because it contains a term that cannot be determined for general T-meshes. Finally, Li et al.~\cite{li20063D} and Zeng and Deng~\cite{zeng2018} discussed the dimension formulas in three-dimensional (3D) cases.

In contrast, several examples of T-meshes, over which the dimension of the spline space was unstable, i.e., the dimension depends on not only the topological information of the T-mesh but also the geometric information of the T-mesh, were provided in~\cite{Ins2012,tc2015,huang2024stability,tc2017,Ins2011}. The above studies show that obtaining a general dimension formula containing only the terms associated with some topological information of the T-mesh is challenging or impossible. Furthermore, numerous researchers have studied the dimension of the spline space over some particular types of T-meshes, such as weighted T-meshes~\cite{dim2014}, diagonalizable T-meshes~\cite{dim2016}, and hierarchical T-meshes~\cite{jin2013,wu2012,zeng2015}. Among these special types of T-meshes, hierarchical T-meshes are the most practical.

Over the past two decades, splines \rev{over} hierarchical T-meshes have been successfully developed in various applications, including surface modeling~\cite{surface,li2010polynomial,wang2011parallel}, solving numerical partial differential equations (NPDEs), and mechanical engineering problems~\cite{nguyen2011,tian2011,HB2,wang2011,xu2015two}. With respect to the dimension calculation, the dimension formula of the PHT-splines proposed in~\cite{surface} is derived directly from~\cite{dim2006}. Since the smoothness order of the PHT-spline is only $\mathbf r=(1,1)$, various researchers have turned to study the spline space with a higher order smoothness over a hierarchical T-mesh. Among these studies, splines with the highest order smoothness, i.e., $\mathbf r=(d_1-1,d_2-1)$, are the most popular, and \rev{when $d_1=d_2=d$,} we denote the spline space by $S_d(\mathscr T)$. Deng et al.~\cite{jin2013} developed a method for calculating the dimensions of spline spaces over T-meshes with a hierarchical structure, namely, the space-embedding method. The study provided the dimension formula for $S_2(\mathscr T)$ over a hierarchical T-mesh $\mathscr T$. Subsequently, Wu et al.~\cite{wu2012} proposed a special type of hierarchical T-meshes (called $(m,n)$-subdivision T-meshes) and obtained a dimension formula using the homology method. Recently, Zeng et al.~\cite{zeng2015} studied the biquadratic and bicubic spline spaces over hierarchical T-meshes with few restrictions using the smoothing cofactor method. The study provided the dimension formulas for these different types of hierarchical T-meshes.

In this paper, we are interested in the dimension of the spline space $S_d(\mathscr T)$ with the highest order of smoothness over a hierarchical T-mesh $\mathscr T$ for arbitrary degree $d$. In this regard, we first propose the concept of a tensor product T-connected component. In particular, considering the tensor product T-connected component as a method of subdivision, it is a generalization of many subdivision modes. Second, the dimension formula of the conformality vector space over the tensor product T-connected \rev{component} is developed using the smoothing cofactor method. Subsequently, the proof that the dimension of the conformality vector space over a T-connected component of a hierarchical T-mesh under tensor product subdivision can be calculated level by level is presented. After that, we propose a special type of hierarchical T-mesh under cross-subdivision based on the concept of the tensor product T-connected component, and the dimensions of $S_d(\mathscr T)$ are given over such a hierarchical T-mesh. Finally, a conjecture regarding the relationship between the hierarchical T-mesh and its crossing-vertex-relationship graph (CVR graph in \rev{short}) is proved. This paper represents a generalization of the results presented in~\cite{zeng2015} and contains the results related to the dimension of the hierarchical B-spline space in~\cite{HB4}.

The main contributions of this paper can be summarized as follows:
\begin{enumerate}
\item[(1)] We propose a generalized concept called the tensor product T-connected component. A dimensional formula of the conformality vector space over the tensor product T-connected component is provided.
\item[(2)] We prove that the dimension of the conformality vector space over a T-connected component of a hierarchical T-mesh under the tensor product subdivision can be calculated level by level.
\item[(3)] A special type of hierarchical T-mesh $\mathscr T$ under cross-subdivision is proposed. We provide the dimension formula of $S_d(\mathscr T)$ and prove a conjecture regarding the relationship between the hierarchical T-mesh and its CVR graph for such T-meshes.
\end{enumerate}

The rest of this paper is organized as follows. In Section~2, we recall some basic notions regarding T-meshes, hierarchical T-meshes, and spline spaces over T-meshes. In Section~3, we propose the concept of a tensor product T-connected component and provide the dimensional formula of the conformality vector space over the tensor product T-connected component. Subsequently, we prove that the dimension of the conformality vector space over a T-connected component of a hierarchical T-mesh under tensor product subdivision can be calculated level by level. In Section~4, we propose a special type of hierarchical T-mesh $\mathscr T$ under cross-subdivision, and present the dimension formula of $S_d(\mathscr T)$. Subsequently, we provide a strategy to modify a hierarchical T-mesh into a special type of hierarchical T-mesh. At last, in Section~5, we prove the conjecture regarding the relationship between the hierarchical T-mesh and its CVR graph. Finally, Section~\ref{sec:conclusion} concludes the paper and discusses future work.

\section{Preliminaries}\label{sec.prelim}
In this section, the \rev{concepts} of T-meshes, hierarchical T-meshes, spline spaces over T-meshes, extended T-meshes, homogeneous boundary conditions, and some preliminary knowledge about the smoothing cofactor method are \rev{recalled}.

\subsection{T-meshes and hierarchical T-meshes}
\begin{definition}[See~\cite{huang2024stability}]
Let $\mathscr T$ be a set of axis-aligned rectangles, and the intersection of any two distinct rectangles in $\mathscr T$ is either empty or consists of points on the boundaries of the \rev{rectangles}. Then $\mathscr T$ is called a \textbf{T-mesh}. Furthermore, if the entire domain occupied by $\mathscr T$ is a simply connected rectangle, then $\mathscr T$ is called a \textbf{regular T-mesh}.
\end{definition}

In this paper, we mainly discuss regular T-meshes. The definitions of vertex, edge and cell of a T-mesh $\mathscr T$ can be found in~\cite{dim2016}. Here, we briefly introduce the above basic concepts by classifying them. Based on the positions, there are two types of vertices---\textbf{boundary vertices} and \textbf{interior vertices}, and two types of edges---\textbf{boundary edges} and \textbf{interior edges}, and if a cell contains at least one boundary edge as its boundary, the cell is called a \textbf{boundary cell}, and otherwise, it \rev{is} called an \textbf{interior cell}. The interior vertex with 3 edges connecting it is called a \textbf{T-node} and the interior vertex with 4 edges connecting it is called a \textbf{cross-vertex}.

The longest-edge (\textbf{$l$-edge} in short) is a line segment that consists of several edges. It is the longest possible line segment whose two endpoints are T-nodes or boundary vertices. There are three types of interior \rev{$l$-edges}: \textbf{cross-cut}, \textbf{ray} and \textbf{T $l$-edge}. These definitions can be found in~\cite{huang2024stability,zeng2015}. Then one can classify the interior vertices \rev{according to the types of $l$-edges that meet at each vertex}: \textbf{mono-vertices} and \textbf{multi-vertices}. If an interior vertex is the intersection of two T $l$-edges, then it is called a multi-vertex. Otherwise, it is called a mono-vertex~\cite{dim20061}.

The union of all the T $l$-edges, including their vertices, is called a \textbf{T-connected component}~\cite{dim20061}, and is denoted by $L(\mathscr T)$. A set of T $l$-edges of $L(\mathscr T)$ including all the vertices on them is called a \textbf{sub-component of $L(\mathscr T)$}. Since a T-connected component can be divided into several connected components and each connected component can be dealt with separately, \rev{in the rest of this paper}, without further declaration, we assume that a T-connected component is connected.

In Figure~\ref{fig1}, a regular T-mesh $\mathscr T$ and its T-connected component are illustrated. Among these marked edges, $v_1v_3$ is a ray, $v_4v_5$ is a cross-cut, and $v_2v_9$ and $v_6v_8$ are two T $l$-edges. The marked interior vertices $v_1,v_2,v_6,v_8$ and $v_9$ are mono-vertices and $v_7$ is the only multi-vertex in the T-mesh. In addition, $v_6v_8$ together with all vertices on it is a sub-component of $L(\mathscr T)$.
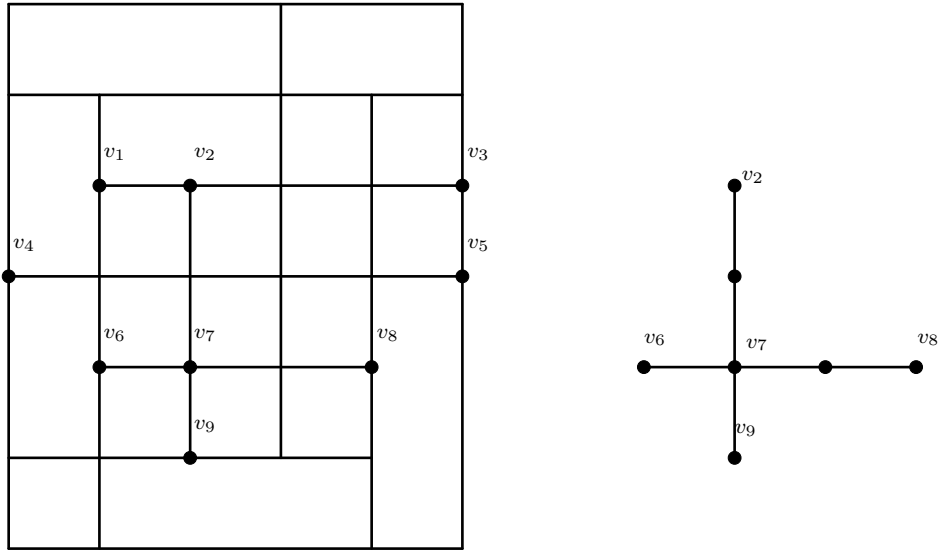
\begin{figure}[htbp]
	\centering
\begin{tikzpicture}[line cap=round,line join=round,>=triangle 45,x=1.0cm,y=1.0cm,scale=1.2]
\draw [line width=1.pt] (1.,1.)-- (6.,1.);
\draw [line width=1.pt] (6.,1.)-- (6.,7.);
\draw [line width=1.pt] (6.,7.)-- (1.,7.);
\draw [line width=1.pt] (1.,7.)-- (1.,1.);
\draw [line width=1.pt] (1.,6.)-- (6.,6.);
\draw [line width=1.pt] (1.,4.)-- (6.,4.);
\draw [line width=1.pt] (2.,6.)-- (2.,1.);
\draw [line width=1.pt] (2.,5.)-- (6.,5.);
\draw [line width=1.pt] (4.,7.)-- (4.,2.);
\draw [line width=1.pt] (1.,2.)-- (5.,2.);
\draw [line width=1.pt] (5.,6.)-- (5.,1.);
\draw [line width=1.pt] (3.,5.)-- (3.,2.);
\draw [line width=1.pt] (2.,3.)-- (5.,3.);
\draw [line width=1.pt] (8.,3.)-- (11.,3.);
\draw [line width=1.pt] (9.,2.)-- (9.,5.);
\begin{scriptsize}
\draw [fill=black] (1.,4.) circle (2.0pt);
\draw[color=black] (1.1727948698137998,4.3475212856370185) node {$v_4$};
\draw [fill=black] (6.,4.) circle (2.0pt);
\draw[color=black] (6.176889747353783,4.3475212856370185) node {$v_5$};
\draw [fill=black] (2.,5.) circle (2.0pt);
\draw[color=black] (2.169987689613434,5.344714105436652) node {$v_1$};
\draw [fill=black] (6.,5.) circle (2.0pt);
\draw[color=black] (6.176889747353783,5.344714105436652) node {$v_3$};
\draw [fill=black] (3.,5.) circle (2.0pt);
\draw[color=black] (3.1671805094130683,5.344714105436652) node {$v_2$};
\draw [fill=black] (3.,2.) circle (2.0pt);
\draw[color=black] (3.1671805094130683,2.353135646037752) node {$v_9$};
\draw [fill=black] (2.,3.) circle (2.0pt);
\draw[color=black] (2.169987689613434,3.3503284658373853) node {$v_6$};
\draw [fill=black] (5.,3.) circle (2.0pt);
\draw[color=black] (5.179696927554149,3.3503284658373853) node {$v_8$};
\draw [fill=black] (3.,3.) circle (2.0pt);
\draw[color=black] (3.1671805094130683,3.3503284658373853) node {$v_7$};
\draw [fill=black] (8.,3.) circle (2.0pt);
\draw[color=black] (8.125948440598524,3.3050015194828566) node {$v_6$};
\draw [fill=black] (11.,3.) circle (2.0pt);
\draw[color=black] (11.135657678539237,3.3050015194828566) node {$v_8$};
\draw [fill=black] (9.,2.) circle (2.0pt);
\draw[color=black] (9.123141260398159,2.3078086996832234) node {$v_9$};
\draw [fill=black] (9.,5.) circle (2.0pt);
\draw[color=black] (9.2,5.1) node {$v_2$};
\draw [fill=black] (9.,3.) circle (2.0pt);
\draw [fill=black] (9.,4.) circle (2.0pt);
\draw [fill=black] (10.,3.) circle (2.0pt);
\draw[color=black] (9.25,3.25) node {$v_7$};
\end{scriptsize}
\end{tikzpicture}
\caption{\label{fig1}A T-mesh and its T-connected component}
\end{figure}

A \textbf{hierarchical T-mesh} is a special type of T-mesh that has a natural level structure~\cite{jin2013,wu2012,zeng2015}. It is defined in a recursive manner. Generally, we start from a tensor-product mesh (level~0), then some cells of level $k$ are each divided under a subdivision mode, where the new cells, new edges, and new vertices are of level $k+1$. The maximal level number that appears is called the \textbf{level} of the hierarchical T-mesh $\mathscr T$, and is denoted by $\lev(\mathscr T)$. For a hierarchical T-mesh $\mathscr T$, we use $T_k$ to denote the set of all of the level $k$ T $l$-edges and the level $k$ vertices of $\mathscr T$, and $N_k$ to denote the set of all the level $k+1$ to level $\lev(\mathscr T)$ T $l$-edges and the level $k+1$ to level $\lev(\mathscr T)$ vertices of $\mathscr T$. In this paper, we mainly discuss the cross subdivision mode, which is a subdivision manner dividing a cell into $2\times2$ subcells equally.

In Figure~\ref{fig2}, a hierarchical T-mesh $\mathscr T$ under the cross subdivision mode and the sets $T_1$ and $N_1$ are illustrated. For $\mathscr T$, $\lev(\mathscr T)=2$. Among these marked edges, $AB$ is an edge in level~1 and $CD$ is an edge in level~2. All vertices marked with black \rev{circles} are the vertices in level~1 and marked with black $\times$ are the vertices in level~2. $T_1$ and $N_1$ are shown in the remaining two subfigures.
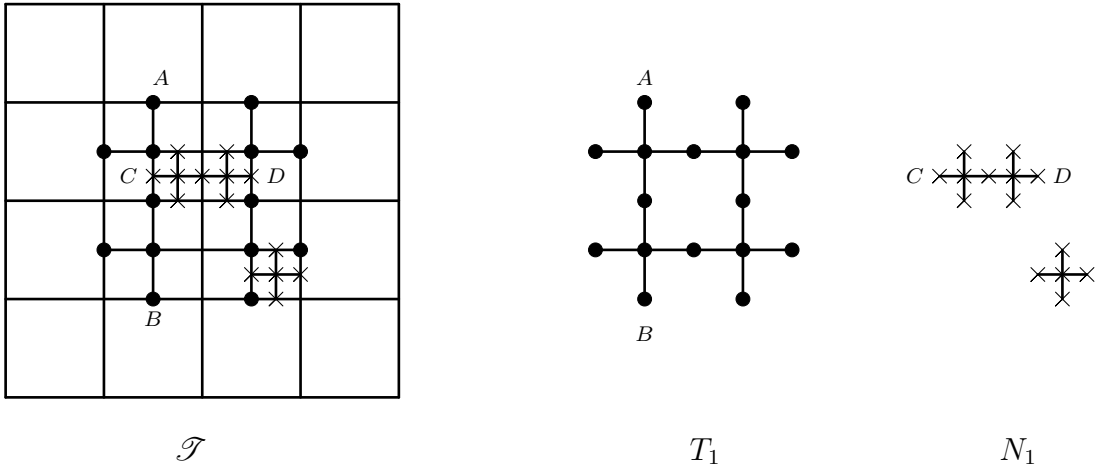
\begin{figure}[htbp]
    \centering
\begin{tikzpicture}[line cap=round,line join=round,>=triangle 45,x=1.0cm,y=1.0cm,scale=1.3]
\draw [line width=1.pt] (1.,1.)-- (5.,1.);
\draw [line width=1.pt] (5.,1.)-- (5.,5.);
\draw [line width=1.pt] (5.,5.)-- (1.,5.);
\draw [line width=1.pt] (1.,5.)-- (1.,1.);
\draw [line width=1.pt] (1.,4.)-- (5.,4.);
\draw [line width=1.pt] (1.,3.)-- (5.,3.);
\draw [line width=1.pt] (1.,2.)-- (5.,2.);
\draw [line width=1.pt] (2.,1.)-- (2.,5.);
\draw [line width=1.pt] (3.,5.)-- (3.,1.);
\draw [line width=1.pt] (4.,1.)-- (4.,5.);
\draw [line width=1.pt] (2.,3.5)-- (4.,3.5);
\draw [line width=1.pt] (2.,2.5)-- (4.,2.5);
\draw [line width=1.pt] (2.5,4.)-- (2.5,2.);
\draw [line width=1.pt] (3.5,4.)-- (3.5,2.);
\draw [line width=1.pt] (3.5,2.25)-- (4.,2.25);
\draw [line width=1.pt] (3.75,2.)-- (3.75,2.5);
\draw [line width=1.pt] (2.5,3.25)-- (3.5,3.25);
\draw [line width=1.pt] (2.75,3.5)-- (2.75,3.);
\draw [line width=1.pt] (3.25,3.5)-- (3.25,3.);
\draw [line width=1.pt] (7,3.5)-- (9,3.5);
\draw [line width=1.pt] (7,2.5)-- (9,2.5);
\draw [line width=1.pt] (7.5,4)-- (7.5,2);
\draw [line width=1.pt] (8.5,2)-- (8.5,4);
\draw [line width=1.pt] (10.5,3.25)-- (11.5,3.25);
\draw [line width=1.pt] (10.75,3)-- (10.75,3.5);
\draw [line width=1.pt] (11.25,3)-- (11.25,3.5);
\draw [line width=1.pt] (11.5,2.25)-- (12,2.25);
\draw [line width=1.pt] (11.75,2)-- (11.75,2.5);
\draw (2.6,0.7) node[anchor=north west] {$\mathscr{T}$};
\draw (7.85,0.7) node[anchor=north west] {$T_1$};
\draw (11,0.7) node[anchor=north west] {$N_1$};
\begin{scriptsize}
\draw [fill=black] (7,3.5) circle (2.0pt);
\draw [fill=black] (9,3.5) circle (2.0pt);
\draw [fill=black] (7,2.5) circle (2.0pt);
\draw [fill=black] (9,2.5) circle (2.0pt);
\draw [fill=black] (7.5,4) circle (2.0pt);
\draw[color=black] (7.5,4.25) node {$A$};
\draw [fill=black] (7.5,2) circle (2.0pt);
\draw[color=black] (7.5,1.65) node {$B$};
\draw [fill=black] (8.5,2) circle (2.0pt);
\draw [fill=black] (8.5,4) circle (2.0pt);
\draw [fill=black] (7.5,3.5) circle (2.0pt);
\draw [fill=black] (8.5,3.5) circle (2.0pt);
\draw [fill=black] (7.5,2.5) circle (2.0pt);
\draw [fill=black] (7.5,3) circle (2.0pt);
\draw [fill=black] (8.5,2.5) circle (2.0pt);
\draw [fill=black] (8.5,3) circle (2.0pt);
\draw [fill=black] (8,3.5) circle (2.0pt);
\draw [fill=black] (8,2.5) circle (2.0pt);
\draw [color=black] (2.75,3.5)-- ++(-2.0pt,-2.0pt) -- ++(4.0pt,4.0pt) ++(-4.0pt,0) -- ++(4.0pt,-4.0pt);
\draw [color=black] (2.75,3.)-- ++(-2.0pt,-2.0pt) -- ++(4.0pt,4.0pt) ++(-4.0pt,0) -- ++(4.0pt,-4.0pt);
\draw [color=black] (2.5,3.25)-- ++(-2.0pt,-2.0pt) -- ++(4.0pt,4.0pt) ++(-4.0pt,0) -- ++(4.0pt,-4.0pt);
\draw[color=black] (2.25,3.25) node {$C$};
\draw [color=black] (3.5,3.25)-- ++(-2.0pt,-2.0pt) -- ++(4.0pt,4.0pt) ++(-4.0pt,0) -- ++(4.0pt,-4.0pt);
\draw[color=black] (3.75,3.25) node {$D$};
\draw [color=black] (3.25,3.5)-- ++(-2.0pt,-2.0pt) -- ++(4.0pt,4.0pt) ++(-4.0pt,0) -- ++(4.0pt,-4.0pt);
\draw [color=black] (3.25,3.)-- ++(-2.0pt,-2.0pt) -- ++(4.0pt,4.0pt) ++(-4.0pt,0) -- ++(4.0pt,-4.0pt);
\draw [color=black] (3.,3.25)-- ++(-2.0pt,-2.0pt) -- ++(4.0pt,4.0pt) ++(-4.0pt,0) -- ++(4.0pt,-4.0pt);
\draw [color=black] (3.75,2.5)-- ++(-2.0pt,-2.0pt) -- ++(4.0pt,4.0pt) ++(-4.0pt,0) -- ++(4.0pt,-4.0pt);
\draw [color=black] (3.5,2.25)-- ++(-2.0pt,-2.0pt) -- ++(4.0pt,4.0pt) ++(-4.0pt,0) -- ++(4.0pt,-4.0pt);
\draw [color=black] (3.75,2.)-- ++(-2.0pt,-2.0pt) -- ++(4.0pt,4.0pt) ++(-4.0pt,0) -- ++(4.0pt,-4.0pt);
\draw [color=black] (4.,2.25)-- ++(-2.0pt,-2.0pt) -- ++(4.0pt,4.0pt) ++(-4.0pt,0) -- ++(4.0pt,-4.0pt);
\draw [fill=black] (2.5,4.) circle (2.0pt);
\draw[color=black] (2.5791889022411962,4.251861774932848) node {$A$};
\draw [fill=black] (2.5,2.) circle (2.0pt);
\draw[color=black] (2.5,1.8) node {$B$};
\draw [fill=black] (3.5,4.) circle (2.0pt);
\draw [fill=black] (4.,3.5) circle (2.0pt);
\draw [fill=black] (2.,3.5) circle (2.0pt);
\draw [fill=black] (2.,2.5) circle (2.0pt);
\draw [fill=black] (4.,2.5) circle (2.0pt);
\draw [fill=black] (3.5,2.) circle (2.0pt);
\draw [fill=black] (2.5,3.5) circle (2.0pt);
\draw [fill=black] (3.5,3.5) circle (2.0pt);
\draw [fill=black] (3.5,3.) circle (2.0pt);
\draw [fill=black] (2.5,3.) circle (2.0pt);
\draw [fill=black] (2.5,2.5) circle (2.0pt);
\draw [fill=black] (3.5,2.5) circle (2.0pt);
\draw [color=black] (10.75,3.5)-- ++(-2.0pt,-2.0pt) -- ++(4.0pt,4.0pt) ++(-4.0pt,0) -- ++(4.0pt,-4.0pt);
\draw [color=black] (10.75,3)-- ++(-2.0pt,-2.0pt) -- ++(4.0pt,4.0pt) ++(-4.0pt,0) -- ++(4.0pt,-4.0pt);
\draw [color=black] (11.25,3.5)-- ++(-2.0pt,-2.0pt) -- ++(4.0pt,4.0pt) ++(-4.0pt,0) -- ++(4.0pt,-4.0pt);
\draw [color=black] (11.25,3)-- ++(-2.0pt,-2.0pt) -- ++(4.0pt,4.0pt) ++(-4.0pt,0) -- ++(4.0pt,-4.0pt);
\draw [color=black] (11.75,2.5)-- ++(-2.0pt,-2.0pt) -- ++(4.0pt,4.0pt) ++(-4.0pt,0) -- ++(4.0pt,-4.0pt);
\draw [color=black] (11.75,2)-- ++(-2.0pt,-2.0pt) -- ++(4.0pt,4.0pt) ++(-4.0pt,0) -- ++(4.0pt,-4.0pt);
\draw [color=black] (11.5,2.25)-- ++(-2.0pt,-2.0pt) -- ++(4.0pt,4.0pt) ++(-4.0pt,0) -- ++(4.0pt,-4.0pt);
\draw [color=black] (12,2.25)-- ++(-2.0pt,-2.0pt) -- ++(4.0pt,4.0pt) ++(-4.0pt,0) -- ++(4.0pt,-4.0pt);
\draw [color=black] (10.5,3.25)-- ++(-2.0pt,-2.0pt) -- ++(4.0pt,4.0pt) ++(-4.0pt,0) -- ++(4.0pt,-4.0pt);
\draw[color=black] (10.25,3.25) node {$C$};
\draw [color=black] (11.5,3.25)-- ++(-2.0pt,-2.0pt) -- ++(4.0pt,4.0pt) ++(-4.0pt,0) -- ++(4.0pt,-4.0pt);
\draw [color=black] (2.75,3.25)-- ++(-2.0pt,-2.0pt) -- ++(4.0pt,4.0pt) ++(-4.0pt,0) -- ++(4.0pt,-4.0pt);
\draw [color=black] (3.25,3.25)-- ++(-2.0pt,-2.0pt) -- ++(4.0pt,4.0pt) ++(-4.0pt,0) -- ++(4.0pt,-4.0pt);
\draw [color=black] (3.75,2.25)-- ++(-2.0pt,-2.0pt) -- ++(4.0pt,4.0pt) ++(-4.0pt,0) -- ++(4.0pt,-4.0pt);
\draw[color=black] (11.75,3.25) node {$D$};
\draw [color=black] (11,3.25)-- ++(-2.0pt,-2.0pt) -- ++(4.0pt,4.0pt) ++(-4.0pt,0) -- ++(4.0pt,-4.0pt);
\draw [color=black] (10.75,3.25)-- ++(-2.0pt,-2.0pt) -- ++(4.0pt,4.0pt) ++(-4.0pt,0) -- ++(4.0pt,-4.0pt);
\draw [color=black] (11.75,2.25)-- ++(-2.0pt,-2.0pt) -- ++(4.0pt,4.0pt) ++(-4.0pt,0) -- ++(4.0pt,-4.0pt);
\draw [color=black] (11.25,3.25)-- ++(-2.0pt,-2.0pt) -- ++(4.0pt,4.0pt) ++(-4.0pt,0) -- ++(4.0pt,-4.0pt);
\end{scriptsize}
\end{tikzpicture}
    \caption{\label{fig2}A hierarchical T-mesh and T-connected component in each level}
\end{figure}

\subsection{Spline spaces over T-meshes and homogeneous boundary conditions}
\begin{definition}[See~\cite{huang2024stability,zeng2015}]
Given a T-mesh $\mathscr T$, let $\mathscr F$ be the set of all the cells in $\mathscr T$ and $\Omega$ be the region occupied by the cells in $\mathscr T$, $\mathbf d=(d_1,d_2)$ and $\mathbf r=(r_1,r_2)$. Then the spline space of degree $\mathbf d$ with smoothness order $\mathbf r$ over $\mathscr T$ is defined as
\[
S_{\mathbf d}^{\mathbf r}(\mathscr T)=\{f(x,y)\in C^{\mathbf r}(\Omega)\mid f(x,y)|_\phi\in P_{\mathbf d},\ \forall\phi\in\mathscr F\},
\]
where $P_{\mathbf d}$ is the function space of all the polynomials with bi-degree \rev{at most} $(d_1,d_2)$, and $C^{\mathbf r}(\Omega)$ is the space consisting of all the bivariate functions in $\Omega$ with $C^{r_1}$ continuity along the $x$-direction and with $C^{r_2}$ continuity along the $y$-direction.
\end{definition}

If $(r_1,r_2)=(d_1-1,d_2-1)$, then the spline space is called the \textbf{highest order of smoothness spline space}. Especially, if $d_1=d_2=d$ and $r_1=r_2=d-1$, the highest order of smoothness spline space is denoted by $S_d(\mathscr T)$. In this paper, we mainly consider the highest order of smoothness spline space with bi-degree $(d,d)$.

In the following, we define two basic concepts: extended T-mesh and spline space with homogeneous boundary condition~\cite{jin2013}. These two concepts play an important role in calculating the dimension of the spline space over a T-mesh with a hierarchical structure.

\begin{definition}[See~\cite{jin2013,zeng2015}]
Given a T-mesh $\mathscr T$, consider the spline space $S_d(\mathscr T)$. The \textbf{extended T-mesh} is an enlarged T-mesh by copying each horizontal boundary $l$-edge and each vertical boundary $l$-edge of $\mathscr T$ $d$ times, and by extending all of the $l$-edges with an endpoint on the boundary of $\mathscr T$, and is denoted by $\bar{\mathscr T}$.
\end{definition}

Figure~\ref{fig3} illustrates an example of a T-mesh and its extended T-mesh for $S_3(\mathscr T)$.

\begin{definition}
Given a T-mesh $\mathscr T$, a spline space related \rev{to} $S_d(\mathscr T)$:
\[
\bar S_d(\mathscr T)=\{f(x,y)\in C^{d-1,d-1}(\mathbb R^2)\mid f(x,y)|_\phi\in P_{\mathbf d},\ \forall\phi\in\mathscr F,\ f|_{\mathbb R^2\setminus\Omega}\equiv0\}
\]
is called the \textbf{spline space with homogeneous boundary condition}.
\end{definition}

An important result in~\cite{jin2013} is that the spline space $S_d(\mathscr T)$ and the spline space $\bar S_d(\bar{\mathscr T})$ have a \rev{close} connection. Specifically, we have the following theorem.

\begin{theorem}[See~\cite{jin2013}]\label{thm2.1}
Given a T-mesh $\mathscr T$, let $\bar{\mathscr T}$ be the extended T-mesh associated with $S_d(\mathscr T)$. Then $S_d(\mathscr T)=\bar S_d(\bar{\mathscr T})|_\Omega$ and $\dim S_d(\mathscr T)=\dim\bar S_d(\bar{\mathscr T})$.
\end{theorem}
\begin{figure}[htbp]
    \centering
    \begin{tikzpicture}[line cap=round,line join=round,>=triangle 45,x=1.0cm,y=1.0cm]
\draw [line width=1.pt] (1.,1.)-- (4.,1.);
\draw [line width=1.pt] (4.,1.)-- (4.,4.);
\draw [line width=1.pt] (4.,4.)-- (1.,4.);
\draw [line width=1.pt] (1.,4.)-- (1.,1.);
\draw [line width=1.pt] (1.,3.)-- (4.,3.);
\draw [line width=1.pt] (3.,4.)-- (3.,1.);
\draw [line width=1.pt] (2.,1.)-- (2.,3.);
\draw [line width=1.pt] (2.,2.)-- (3.,2.);
\draw [line width=1.pt] (5.4,0.4)-- (9.6,0.4);
\draw [line width=1.pt] (9.6,0.4)-- (9.6,4.6);
\draw [line width=1.pt] (9.6,4.6)-- (5.4,4.6);
\draw [line width=1.pt] (5.4,4.6)-- (5.4,0.4);
\draw [line width=1.pt] (5.4,4)-- (9.6,4);
\draw [line width=1.pt] (5.4,1)-- (9.6,1);
\draw [line width=1.pt] (6,0.4)-- (6,4.6);
\draw [line width=1.pt] (9,4.6)-- (9,0.4);
\draw [line width=1.pt] (5.4,4.4)-- (9.6,4.4);
\draw [line width=1.pt] (5.4,4.2)-- (9.6,4.2);
\draw [line width=1.pt] (5.4,0.8)-- (9.6,0.8);
\draw [line width=1.pt] (5.4,0.6)-- (9.6,0.6);
\draw [line width=1.pt] (5.6,4.6)-- (5.6,0.4);
\draw [line width=1.pt] (5.8,4.6)-- (5.8,0.4);
\draw [line width=1.pt] (9.2,4.6)-- (9.2,0.4);
\draw [line width=1.pt] (9.4,4.6)-- (9.4,0.4);
\draw [line width=1.pt] (8,4.6)-- (8,0.4);
\draw [line width=1.pt] (5.4,3)-- (9.6,3);
\draw [line width=1.pt] (7,3)-- (7,0.4);
\draw [line width=1.pt] (7,2)-- (8,2);
\draw (2.3,-0.05) node[anchor=north west] {$\mathscr{T}$};
\draw (7.3,-0.05) node[anchor=north west] {$\bar{\mathscr{T}}$};
\end{tikzpicture}
    \caption{\label{fig3}A T-mesh and its extended T-mesh for $S_3(\mathscr T)$}
\end{figure}
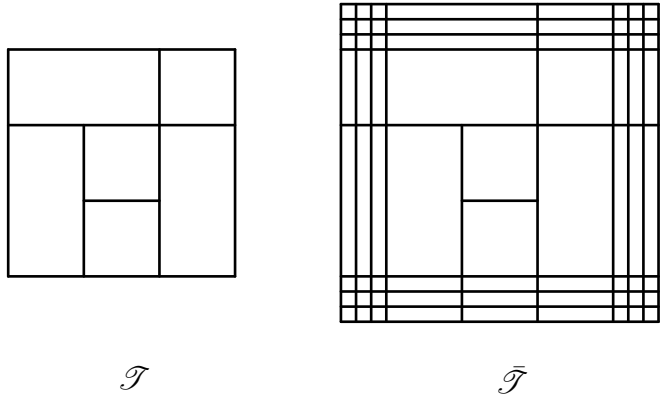

\subsection{The smoothing cofactor method}
The smoothing cofactor method is a general method to deal with the dimension of the spline space, and it was first introduced in~\cite{sm1,sm2}. Recently, the method was applied to calculate the dimension of the spline space over a general T-mesh~\cite{huang2024stability,wu2012,zeng2015}.
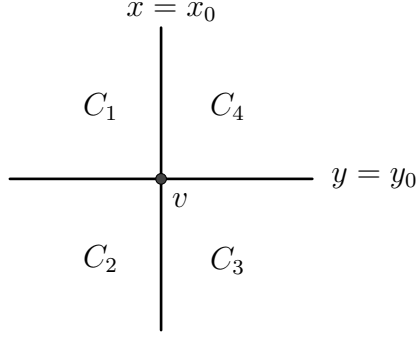
\begin{figure}[htbp]
	\centering
	\definecolor{uuuuuu}{rgb}{0.26666666666666666,0.26666666666666666,0.26666666666666666}
	\begin{tikzpicture}[line cap=round,line join=round,>=triangle 45,x=1cm,y=1cm]
		\draw [line width=1pt] (4,2.5)-- (8,2.5);
		\draw [line width=1pt] (6,4.5)-- (6,0.5);
		\draw (4.824242424242416,3.7761616161616014) node[anchor=north west] {$C_1$};
		\draw (4.824242424242416,1.764040404040395) node[anchor=north west] {$C_2$};
		\draw (6.503030303030293,1.7397979797979708) node[anchor=north west] {$C_3$};
		\draw (6.503030303030293,3.7761616161616014) node[anchor=north west] {$C_4$};
		\draw (8.110101010100999,2.7822222222222104) node[anchor=north west] {$y=y_0$};
		\draw (5.4,4.980202020202002) node[anchor=north west] {$x=x_0$};
		\draw (6.004040404040396,2.4610101010100897) node[anchor=north west] {$v$};
		\begin{scriptsize}
			\draw [fill=uuuuuu] (6,2.5) circle (2pt);
		\end{scriptsize}
	\end{tikzpicture}
	\caption{\label{fig4}Local conformality condition}
\end{figure}
Referring to Figure~\ref{fig4}, let $s(x,y)\in S_d(\mathscr T)$ be a spline function over a T-mesh $\mathscr T$, and $v$ be an interior vertex of $\mathscr T$ intersected by two $l$-edges along the $x$-direction and $y$-direction, respectively. Let $C_i$ be the four cells surrounding the vertex $v$ and $s_i(x,y)=s(x,y)|_{C_i}$, $i=1,2,3,4$. Assume that the equations of these two T $l$-edges are $y=y_0$ and $x=x_0$, respectively. Then we have
\begin{align}\label{eq1}
s_2(x,y)&=s_1(x,y)+\gamma_1(x)(y-y_0)^d,\nonumber\\
s_4(x,y)&=s_1(x,y)+\gamma_2(y)(x-x_0)^d,\\
s_3(x,y)&=s_1(x,y)+\gamma_1(x)(y-y_0)^d+\gamma_2(y)(x-x_0)^d+\delta(x-x_0)^d(y-y_0)^d,\nonumber
\end{align}
where $\delta\in\mathbb R$ is called the \textbf{vertex cofactor} of $s(x,y)$ corresponding to the vertex $v$. The above equations are called the \textbf{local conformality condition}. The details can be found in~\cite{wu2012}.

The \textbf{global conformality condition} gives relations among all the vertex cofactors in a T $l$-edge. Referring to Figure~\ref{fig5}, suppose that a T $l$-edge has $r$ vertices whose $x$-coordinates are $x_1,\ldots,x_r$ and the corresponding vertex cofactors are $\delta_1,\ldots,\delta_r$, respectively. Then we have (see~\cite{dim20061,dim2016,wu2012})
\[
\sum_{i=1}^r\delta_i(x-x_i)^d=0.
\]
This equation is equivalent to a linear system (denoted by $\mathscr P=0$)
\begin{equation}\label{tl}
\begin{pmatrix}
1&1&\cdots&\cdots&1\\
x_1&x_2&\cdots&\cdots&x_r\\
x_1^2&x_2^2&\cdots&\cdots&x_r^2\\
\vdots&\vdots&\vdots&\vdots&\vdots\\
x_1^{d-1}&x_2^{d-1}&\cdots&\cdots&x_r^{d-1}\\
x_1^d&x_2^d&\cdots&\cdots&x_r^d
\end{pmatrix}
\begin{pmatrix}\delta_1\\\delta_2\\\delta_3\\\vdots\\\delta_{r-1}\\\delta_r\end{pmatrix}
=\begin{pmatrix}0\\0\\0\\\vdots\\0\\0\end{pmatrix}.
\end{equation}

Note that when $r<d+1$, the null space of the linear system is zero. Thus, in such a case, the T $l$-edge will not contribute to the dimension of the spline space. We call such a T $l$-edge a \textbf{vanished edge}. We assume that vanished edges do not exist throughout the paper.

The linear systems determined by all the T $l$-edges of $\mathscr T$ are called the global conformality condition of $S_d(\mathscr T)$.

\begin{definition}[See~\cite{huang2024stability,zeng2015}]\label{def2.3}
Let $L(\mathscr T)$ be the T-connected component of $\mathscr T$, $l_1,l_2,\ldots,l_t$ be all the T $l$-edges of $L(\mathscr T)$, and $\delta_1,\delta_2,\ldots,\delta_v$ be all the vertex cofactors of $L(\mathscr T)$. The conformality vector space (\textsf{CVS} in short) of $L(\mathscr T)$ is defined by
\[
\textsf{CVS}[L(\mathscr T)]:=\{\boldsymbol\delta=(\delta_1,\delta_2,\ldots,\delta_v)\mid\mathscr P_{l_i}=0,\ 1\le i\le t\}.
\]
The coefficient matrix for the homogeneous linear equations of $\textsf{CVS}[L(\mathscr T)]$ is called the \textbf{conformality matrix} associated with $L(\mathscr T)$, and is denoted by $M(L(\mathscr T))$.
\end{definition}
\begin{figure}[htbp]
    \centering
\definecolor{uuuuuu}{rgb}{0.26666666666666666,0.26666666666666666,0.26666666666666666}
\begin{tikzpicture}[line cap=round,line join=round,>=triangle 45,x=1cm,y=1cm,scale=2]
\draw [line width=1pt] (1,0.5)-- (1,2.5);
\draw [line width=1pt] (1,1.5)-- (5,1.5);
\draw [line width=1pt] (5,2.5)-- (5,0.5);
\draw [line width=1pt] (1.5,1.5)-- (1.5,2);
\draw [line width=1pt] (2,1.5)-- (2,1);
\draw [line width=1pt] (4.5,1.5)-- (4.5,1);
\draw [line width=1pt] (4,1.5)-- (4,2);
\draw (0.74,1.5) node[anchor=north west] {$\delta_1$};
\draw (4.95,1.5) node[anchor=north west] {$\delta_r$};
\draw (2.5,1.5) node[anchor=north west] {$\ldots\ldots\ldots$};
\begin{scriptsize}
\draw [fill=uuuuuu] (1,1.5) circle (1pt);
\draw [fill=uuuuuu] (1.5,1.5) circle (1pt);
\draw [fill=uuuuuu] (2,1.5) circle (1pt);
\draw [fill=black] (4,1.5) circle (1pt);
\draw [fill=uuuuuu] (4.5,1.5) circle (1pt);
\draw [fill=black] (5,1.5) circle (1pt);
\end{scriptsize}
\end{tikzpicture}
    \caption{\label{fig5}Global conformality condition along a horizontal T $l$-edge}
\end{figure}
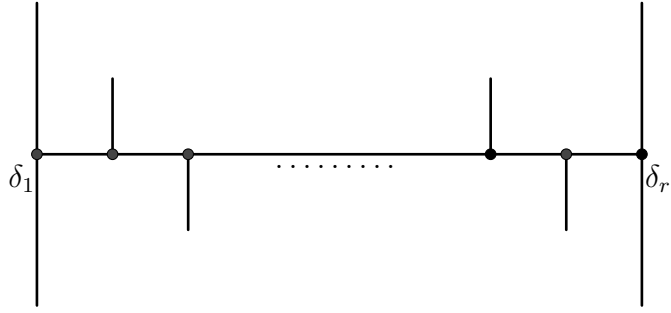

Notice that the concepts of the conformality vector space and the conformality matrix can be generalized to any sub-component of $L(\mathscr T)$. On the other hand, we have the following proposition, which is a main result of the smoothing cofactor method applied to T-meshes.

\begin{proposition}[See~\cite{zeng2015}]\label{prop2.1}
Given a T-mesh $\mathscr T$, and suppose that $M(L(\mathscr T))$ is the conformality matrix associated with the T-connected component $L(\mathscr T)$. Then
\[
\dim S_d(\mathscr T)=(d+1)^2+c(d+1)+n_v+\dim\textsf{CVS}[L(\mathscr T)],
\]
where $c$ is the number of cross-cuts of $\mathscr T$, and $n_v$ is the number of all the interior vertices of $\mathscr T$ with all vertices on T $l$-edges being removed.
\end{proposition}

The above Proposition~\ref{prop2.1} shows that the dimension of the spline space over a T-mesh $\mathscr T$ is indeed dependent on the dimension of the conformality vector space of $L(\mathscr T)$.

\subsection{The bipartite partition and diagonalizable component}
To calculate the dimension of the conformality vector space of the T-connected component for a given T-mesh, one can split the T-connected component into two parts.

\begin{definition}[See~\cite{huang2024stability}]
Let $L(\mathscr T)$ be the T-connected component of a T-mesh $\mathscr T$ and $L\subset L(\mathscr T)$ be a sub-component of $L(\mathscr T)$. Then we call $\{L,L(\mathscr T)\setminus L\}$ a \textbf{bipartite partition} of $L(\mathscr T)$.
\end{definition}

Figure~\ref{fig6} illustrates a bipartite partition of the T-connected component $L(\mathscr T)$ in Figure~\ref{fig1}. Notice that $L(\mathscr T)\setminus L=\{v_6v_8\}$ may not be a sub-component of $L(\mathscr T)$ since the vertex $v_7$ that lies on both $L$ and $L(\mathscr T)\setminus L$ is removed from the edge $v_6v_8$.

\begin{definition}[See~\cite{dim2016}]
Let $\{l_1,l_2,\ldots,l_t\}$ be all the T $l$-edges of $L(\mathscr T)$ for a given T-mesh $\mathscr T$. For an ordering of all these T $l$-edges, i.e., $l_1\succ l_2\ldots\succ l_t$, if $r(l_i)\ge d+1$, where $r(l_i)$ is the number of vertices on $l_i$ but not on $l_j$, $j=1,2,\ldots,i-1$, then we call such ordering a \textbf{reasonable order}, and $L(\mathscr T)$ is called a \textbf{diagonalizable T-connected component}. If any order of all the T-edges in $L(\mathscr T)$ is not a reasonable order, then $L(\mathscr T)$ is called a \textbf{non-diagonalizable T-connected component}.
\end{definition}
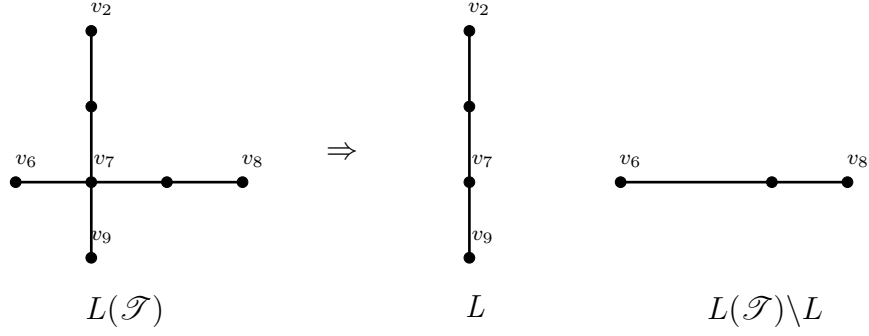
\begin{figure}[htbp]
    \centering
    \begin{tikzpicture}[line cap=round,line join=round,>=triangle 45,x=1.0cm,y=1.0cm]
\draw [line width=1.pt] (1.,2.)-- (4.,2.);
\draw [line width=1.pt] (2.,1.)-- (2.,4.);
\draw [line width=1.pt] (7.,1.)-- (7.,4.);
\draw [line width=1.pt] (9.,2.)-- (12.,2.);
\draw (4.963510987832924,2.6) node[anchor=north west] {$\Rightarrow$};
\draw (1.7803470140463864,0.65) node[anchor=north west] {$L(\mathscr{T})$};
\draw (6.8,0.65) node[anchor=north west] {$\it{L}$};
\draw (10.,0.65) node[anchor=north west] {$L(\mathscr{T})\backslash L$};
\begin{scriptsize}
\draw [fill=black] (1.,2.) circle (2.0pt);
\draw[color=black] (1.1451743862495314,2.284302347583261) node {$v_6$};
\draw [fill=black] (4.,2.) circle (2.0pt);
\draw[color=black] (4.138516655177239,2.284302347583261) node {$v_8$};
\draw [fill=black] (2.,1.) circle (2.0pt);
\draw[color=black] (2.1380879193572584,1.2767871448710042) node {$v_9$};
\draw [fill=black] (2.,4.) circle (2.0pt);
\draw[color=black] (2.1380879193572584,4.2847310834032495) node {$v_2$};
\draw [fill=black] (2.,2.) circle (2.0pt);
\draw[color=black] (2.18,2.284302347583261) node {$v_7$};
\draw [fill=black] (2.,3.) circle (2.0pt);
\draw [fill=black] (3.,2.) circle (2.0pt);
\draw [fill=black] (7.,1.) circle (2.0pt);
\draw[color=black] (7.18,1.2767871448710042) node {$v_9$};
\draw [fill=black] (7.,4.) circle (2.0pt);
\draw[color=black] (7.131858924104947,4.2847310834032495) node {$v_2$};
\draw [fill=black] (9.,2.) circle (2.0pt);
\draw[color=black] (9.132287659924927,2.284302347583261) node {$v_6$};
\draw [fill=black] (12.,2.) circle (2.0pt);
\draw[color=black] (12.14023159845716,2.284302347583261) node {$v_8$};
\draw [fill=black] (7.,3.) circle (2.0pt);
\draw [fill=black] (7.,2.) circle (2.0pt);
\draw[color=black] (7.18,2.284302347583261) node {$v_7$};
\draw [fill=black] (11.,2.) circle (2.0pt);
\end{scriptsize}
\end{tikzpicture}
    \caption{\label{fig6}A bipartite partition of $L(\mathscr T)$}
\end{figure}

For a bipartite partition \rev{$\{L_1,L_2\}$ of $L(\mathscr T)$} with $L_1$ being the sub-component of $L(\mathscr T)$, \rev{the following proposition is stated in~\cite{huang2024stability}}.

\begin{proposition}[See~\cite{huang2024stability}]\label{prop2.2}
Given a T-mesh $\mathscr T$, suppose that \rev{$\{L_1,L_2\}$} is a bipartite partition of $L(\mathscr T)$ with $L_1\subset L(\mathscr T)$. If all the T $l$-edges in $L_2$ have a reasonable order, then
\begin{equation}\label{rea}
\dim\textsf{CVS}[L(\mathscr T)]=\dim\textsf{CVS}[L_1]+\dim\textsf{CVS}[L_2].
\end{equation}
\end{proposition}

Li and Deng~\cite{dim2016} proved that if several T $l$-edges have a reasonable order, then the corresponding conformality matrix has a full row rank. Thus, by Propositions~\ref{prop2.1} and~\ref{prop2.2}, one can further transform the calculation of the dimension of $S_d(\mathscr T)$ into the calculation of the conformality vector space of $L_1$, which is a sub-component of $L(\mathscr T)$.

\begin{definition}[See~\cite{huang2024stability}]
Given a T-mesh $\mathscr T$, if there does not exist a bipartite partition $\{L_1,L_2\}$ of $L(\mathscr T)$ such that all the T-edges of $L_2$ have a reasonable order, then $\mathscr T$ is called a \textbf{completely non-diagonalizable T-mesh}, and $L(\mathscr T)$ is called a \textbf{completely non-diagonalizable T-connected component}.
\end{definition}

The following lemma provided in~\cite{zeng2015} is useful in later proofs.
\begin{lemma}[See~\cite{zeng2015}]\label{lem2.1}
For a given T-mesh $\mathscr T_1$, suppose that \rev{$\{L_1,L_2\}$} is a bipartite partition of $L(\mathscr T_1)$ with $L_1\subset L(\mathscr T_1)$, and $L_1=\{l_1,l_2,\ldots,l_m\}$. One can extend $l_i$ to $l_i'$ for $i=1,2,\ldots,m$ such that $l_i'$ is still a T $l$-edge of the new T-mesh $\mathscr T_2$. Suppose that $L_1'=\{l_1',l_2',\ldots,l_m'\}$ and $\{L_1',L_2\}$ is a bipartite partition of $L(\mathscr T_2)$. If $\dim\textsf{CVS}[L(\mathscr T_2)]=\dim\textsf{CVS}[L_1']+\dim\textsf{CVS}[L_2]$, then
\[
\dim\textsf{CVS}[L(\mathscr T_1)]=\dim\textsf{CVS}[L_1]+\dim\textsf{CVS}[L_2].
\]
\end{lemma}

\section{The tensor product T-connected component}
In this section, we first introduce the concept of the tensor product T-connected component. Then we calculate the dimension of the conformality vector space of a tensor product T-connected component. Finally, we prove that if one level subdivision is a tensor product subdivision, then the dimension of the conformality vector space over the whole T-connected component can be calculated level by level.

\begin{definition}\label{def tp}
Let $L(\mathscr T)$ be the T-connected component of a T-mesh $\mathscr T$. We define $L(\mathscr T)$ as a \textbf{tensor product T-connected component} if it comprises all $l$-edges, possibly excluding some boundary edges, of a tensor product mesh, such that these $l$-edges form an aligned grid of horizontal and vertical lines. \rev{The complete vertex $x$-coordinate sets on horizontal T $l$-edges, and the complete vertex $y$-coordinate sets on vertical T $l$-edges, are respectively identical, including the mono-vertices.}
\end{definition}

An example is illustrated in Figure~\ref{fig7}, where the T-mesh shown on the left has a tensor product T-connected component shown on the right. Note that there may be some mono-vertices on $L(\mathscr T)$ which are the intersections of the cross cuts and rays of $\mathscr T$ with the $l$-edges of $L(\mathscr T)$. Thus, there is the same number of vertices on each horizontal (vertical) T $l$-edge in a tensor product T-connected component, and the set of $x$-coordinates ($y$-coordinates) of the vertices on each horizontal (vertical) T $l$-edge is also the same.
\begin{figure}[htbp]
    \centering
\begin{tikzpicture}[line cap=round,line join=round,>=triangle 45,x=1cm,y=1cm,scale=1.5]
\draw [line width=1pt] (0.5,0.5)-- (3.5,0.5);
\draw [line width=1pt] (3.5,0.5)-- (3.5,2.5);
\draw [line width=1pt] (3.5,2.5)-- (0.5,2.5);
\draw [line width=1pt] (0.5,2.5)-- (0.5,0.5);
\draw [line width=1pt] (0.5,2)-- (3.5,2);
\draw [line width=1pt] (0.5,1)-- (3.5,1);
\draw [line width=1pt] (1,2.5)-- (1,0.5);
\draw [line width=1pt] (3,2.5)-- (3,0.5);
\draw [line width=1pt] (1.5,2)-- (1.5,1);
\draw [line width=1pt] (1,1.5)-- (3,1.5);
\draw [line width=1pt] (2.5,2)-- (2.5,1);
\draw [line width=1pt] (4.5,1.5)-- (6.5,1.5);
\draw [line width=1pt] (5,2)-- (5,1);
\draw [line width=1pt] (6,2)-- (6,1);
\draw [line width=1pt,dotted] (4.5,2)-- (4.5,1);
\draw [line width=1pt,dotted] (4.5,1)-- (6.5,1);
\draw [line width=1pt,dotted] (6.5,1)-- (6.5,2);
\draw [line width=1pt,dotted] (6.5,2)-- (4.5,2);
\begin{scriptsize}
\draw [fill=black] (1.5,2) circle (1pt);
\draw[color=black] (1.6,2.15) node {$v_{1}$};
\draw [fill=black] (1.5,1) circle (1pt);
\draw[color=black] (1.64,1.15) node {$v_{5}$};
\draw [fill=black] (1,1.5) circle (1pt);
\draw[color=black] (1.15,1.64) node {$v_{3}$};
\draw [fill=black] (3,1.5) circle (1pt);
\draw[color=black] (3.15,1.64) node {$v_{4}$};
\draw [fill=black] (2.5,2) circle (1pt);
\draw[color=black] (2.6,2.15) node {$v_{2}$};
\draw [fill=black] (2.5,1) circle (1pt);
\draw[color=black] (2.64,1.15) node {$v_{6}$};
\draw [fill=black] (4.5,1.5) circle (1pt);
\draw[color=black] (4.6,1.6) node {$v_3$};
\draw [fill=black] (6.5,1.5) circle (1pt);
\draw[color=black] (6.6,1.6) node {$v_4$};
\draw [fill=black] (5,2) circle (1pt);
\draw[color=black] (5.13,2.15) node {$v_1$};
\draw [fill=black] (5,1) circle (1pt);
\draw[color=black] (5.13,1.15) node {$v_5$};
\draw [fill=black] (6,2) circle (1pt);
\draw[color=black] (6.13,2.15) node {$v_2$};
\draw [fill=black] (6,1) circle (1pt);
\draw[color=black] (6.13,1.15) node {$v_6$};
\end{scriptsize}
\end{tikzpicture}
    \caption{\label{fig7}A T-mesh with a tensor product T-connected component}
\end{figure}
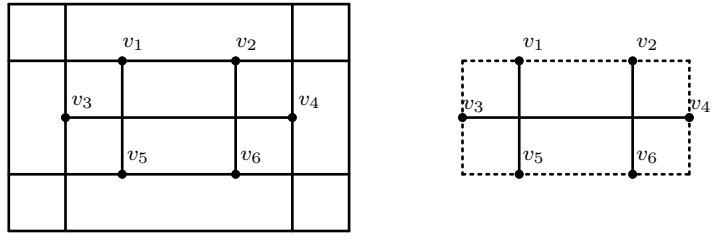

The tensor product T-connected component, as defined in Definition~\ref{def tp}, is a simple class of T-connected components that excludes T-cycle structures, which can cause dimensional instability~\cite{tc2015,tc2017}. Despite its simplicity, this component may include non-diagonalizable cases, complicating the computation of the conformality vector space dimension. The dimension formula for the conformality vector space of a tensor product T-connected component is presented below.

Next, we introduce a basic result about the coefficient matrix of the linear system~\eqref{tl}.

\begin{lemma}[See~\cite{huang2024stability}]\label{lemma3.1}
Let
\[
V_k^n=\begin{pmatrix}1&1&\cdots&1\\s_1&s_2&\cdots&s_k\\\vdots&\vdots&\vdots&\vdots\\s_1^n&s_2^n&\cdots&s_k^n\end{pmatrix}
\]
be a matrix with distinct $s_1,s_2,\ldots,s_k\in\mathbb R$, $k>n+1$. Then the reduced row echelon form of $V_k^n$ is $(I_{n+1},S)$, where
\[
S=\begin{pmatrix}f_1(s_{n+2})&f_1(s_{n+3})&\cdots&f_1(s_k)\\f_2(s_{n+2})&f_2(s_{n+3})&\cdots&f_2(s_k)\\\vdots&\vdots&\vdots&\vdots\\f_{n+1}(s_{n+2})&f_{n+1}(s_{n+3})&\cdots&f_{n+1}(s_k)\end{pmatrix},
\]
and $f_i(x)=\prod_{\substack{j=1\\j\ne i}}^{n+1}\frac{x-s_j}{s_i-s_j}$, $i=1,2,\ldots,n+1$.
\end{lemma}
\begin{proof}
Let $P$ be the matrix consisting of the first $n+1$ columns of $V_k^n$. Obviously, $P$ is invertible and $P^{-1}V_k^n=(I_{n+1},S)$. \hfill$\Box$
\end{proof}

\begin{corollary}[See~\cite{huang2024stability}]\label{cor3.1}
Suppose $k>n+1$, and let $1\le i_1<i_2<\cdots<i_{n+1}\le k$ and $i_1,i_2,\ldots,i_k$ be a permutation of $1,2,\ldots,k$. Then one can perform row elementary transformations such that $V_k^n$ reduces to $W_k^n$, where the submatrix
\[
W_k^n\begin{pmatrix}1&2&\cdots&n+1\\i_1&i_2&\cdots&i_{n+1}\end{pmatrix}=I_{n+1}
\]
and
\[
W_k^n\begin{pmatrix}1&2&\cdots&n+1\\i_{n+2}&i_{n+3}&\cdots&i_k\end{pmatrix}=
\begin{pmatrix}
f_1(s_{i_{n+2}})&f_1(s_{i_{n+3}})&\cdots&f_1(s_{i_k})\\
f_2(s_{i_{n+2}})&f_2(s_{i_{n+3}})&\cdots&f_2(s_{i_k})\\
\vdots&\vdots&\vdots&\vdots\\
f_{n+1}(s_{i_{n+2}})&f_{n+1}(s_{i_{n+3}})&\cdots&f_{n+1}(s_{i_k})
\end{pmatrix}.
\]
\rev{Here $f_j(x)=\displaystyle\prod_{\substack{1\le h\le n+1\\h\ne j}}\frac{x-s_{i_h}}{s_{i_j}-s_{i_h}}$, $j=1,\ldots,n+1$, is defined using the selected pivot nodes.}
\end{corollary}

Now we introduce the main result about the dimension of the conformality vector space of a tensor product T-connected component.

\begin{theorem}\label{thm3.1}
Suppose that $L(\mathscr T)$ is a tensor product T-connected component of a T-mesh $\mathscr T$, then the dimension of the conformality vector space of $L(\mathscr T)$ is
\[
\dim\textsf{CVS}[L(\mathscr T)]=
\begin{cases}
m(p-d-1)+n(q-d-1)\rev{{}-mn},&L(\mathscr T)\text{ is diagonalizable},\\
(p-d-1)(q-d-1),&L(\mathscr T)\text{ is non-diagonalizable},
\end{cases}
\]
where $m$ and $n$ are the numbers of horizontal and vertical T $l$-edges of $L(\mathscr T)$ respectively, $p$ and $q$ are the numbers of vertices on each horizontal and vertical T $l$-edge respectively, and $d$ is the degree of spline functions.
\end{theorem}

\begin{proof}
    Let $l_1,l_2,\ldots,l_m$ be all the horizontal T $l$-edges and $l_{m+1},l_{m+2},\ldots,l_{m+n}$ be all the vertical T $l$-edges of $\mathscr{T}$, respectively.
    Suppose that there are $p$ vertices in each $l_i$, $i=1,2,\ldots,m$, denoted by $v_1^{(i)},v_2^{(i)},\ldots, v_{p}^{(i)}$, among them $v_{1}^{(i)},v_{2}^{(i)},\ldots,v_{\tilde p}^{(i)}$ are the mono-vertices of $l_i$, where $1\le \tilde p\le p$,  and there are $q$ vertices in each $l_j$, $j=m+1,m+2,\ldots,m+n$, denoted by $v_1^{(j)},v_2^{(j)},\ldots, v_{q}^{(j)}$, among them $\rev{v_1^{(j)},v_2^{(j)}},\ldots,v_{\tilde q}^{(j)}$ are the mono-vertices of $l_j$,
    where $1\le \tilde q \le q$.
    Let the $x$-coordinates of $v_1^{(i)},v_2^{(i)},\ldots, v_{p}^{(i)}$ be $s_1,s_2,\ldots,s_p$ respectively, and the $y$-coordinates of $v_1^{(j)},v_2^{(j)},\ldots, v_{q}^{(j)}$ be $t_1,t_2,\ldots,t_q$ respectively. Since each horizontal (vertical) T $l$-edge intersects with all the vertical (horizontal) T $l$-edge, we see that \rev{$n=p-\tilde p$} is the number of multi-vertices on $l_i$, $i=1,2,\ldots, m$, and \rev{$m=q-\tilde q$} is the number of multi-vertices on $l_j$, $j=m+1,m+2,\ldots,m+n$.

    Notice that if $\tilde p\ge d+1$ or $\tilde q\ge d+1$, then the number of mono-vertices of $l_i$, $i=1,2,\ldots,m$ or $l_j$, $j=m+1,m+2,\ldots,m+n$ is greater \rev{than} or equal to $d+1$. Thus $L(\mathscr{T})$ has a reasonable order, and in this case, the dimension of the conformality vector space of $L(\mathscr{T})$ is
    $$\dim\textsf{CVS}[L(\mathscr T)]=mp+nq\rev{{}-mn}-(d+1)(m+n)=m(p-d-1)+n(q-d-1)\rev{{}-mn}.$$
    In the following, we only consider the case of $\tilde q<d+1$ and $\tilde p<d+1$, i.e., $L(\mathscr T)$ is a non-diagonalizable T-connected component.

    Set the ordering of all T $l$-edges as $l_1\succ l_2\succ\ldots\succ l_m\succ l_{m+1}\succ\ldots\succ \rev{l_{m+n}}$, and the ordering of vertices as $v_1^{(i)}\succ v_2^{(i)}\succ\ldots\succ v_{p}^{(i)}$ in each T $l$-edge $l_i$ for $i=1,2,\ldots,m$, and $v_1^{(j)}\succ v_2^{(j)}\succ\ldots\succ v_{\tilde q}^{(j)}\succ\ldots\succ v_{q}^{(j)}$ in each T $l$-edge $l_j$ for $j=m+1,\ldots,m+n$. \rev{The numbering is chosen so that $l_j$ has ordinate $t_{\tilde q+j}$ and $l_{m+i}$ has abscissa $s_{\tilde p+i}$. Each multi-vertex is included only once among the columns.} Then the conformality matrix of $L(\mathscr{T})$ can be written as
    $$M=\begin{pmatrix}A & O \\ C & B\end{pmatrix},$$
    where $A=\mathrm{diag}\underbrace{(A',\ldots,A')}_{m}$,   $B=\mathrm{diag}\underbrace{(B',\ldots,B')}_{n}$ with
    $$A'=\begin{pmatrix}
        1 & 1 & \ldots & 1\\
        s_1 & s_2 & \ldots & s_p\\
        \vdots & \vdots & \vdots & \vdots\\
        s_1^{d} & s_2^{d} & \ldots & s_p^{d}
    \end{pmatrix},\qquad
    B'=\begin{pmatrix}
        1 & 1 & \ldots & 1\\
        t_{1} & t_{2} & \ldots & t_{\tilde q}\\
        \vdots & \vdots & \vdots & \vdots\\
        t_{1}^{d} & t_{2}^{d} & \ldots & t_{\tilde q}^{d}
    \end{pmatrix},$$
and
   $$C=\begin{pmatrix}
C_{11} & C_{12} & \ldots & C_{1m}\\
\vdots & \vdots & \vdots & \vdots\\
C_{n1} & C_{n2} & \ldots & C_{nm}\\
    \end{pmatrix}$$
with $C_{ij}\in\mathbb{R}^{(d+1)\times p}$ in the form
$$\begin{matrix}
   \begin{pmatrix}
     0 & 0 & \ldots & 1 & \ldots & 0\\
     0 & 0 & \ldots & \rev{t_{\tilde q+j}} & \ldots & 0\\
     \vdots & \vdots & \vdots & \vdots & \vdots & \vdots\\
     0 & 0 & \ldots & \rev{t_{\tilde q+j}}^d & \ldots & 0
     \end{pmatrix}\\
     \begin{matrix}
        \hphantom{\vdots}&
        \hphantom{\vdots} &
        \hphantom{\vdots} &
    \uparrow &
    \hphantom{\vdots} & \\
    \hphantom{\vdots}&
        \hphantom{\vdots} &
        \hphantom{\vdots} &
    \text{($\tilde p+i$)-th column} &
    \hphantom{\vdots} &
        \end{matrix}
\end{matrix}.$$

 Next, we perform elementary row transformations on $M$ in block form, and by Lemma~\ref{lemma3.1} and Corollary~\ref{cor3.1}, one can obtain the matrix
 $$\bar{M}=\begin{pmatrix}
     \bar{A} & O\\
     \bar{C} & \bar{B}
     \end{pmatrix},$$
  where $\bar{A}=\mathrm{diag}\underbrace{(\bar{A}',\ldots,\bar{A}')}_{m}$, $\bar{B}=\mathrm{diag}\underbrace{(\bar{B}',\ldots,\bar{B}')}_{n}$. Here   $\bar{A}'=(
      I_{d+1},\rev{S_{\bar A}})$
with $$S_{\bar{A}}=\begin{pmatrix}
    f_1(s_{d+2}) & f_1(s_{d+3}) & \ldots & f_1(s_{p})\\
    f_2(s_{d+2}) & f_2(s_{d+3}) & \ldots & f_2(s_{p})\\
    \vdots & \vdots & \vdots & \vdots\\
    f_{d+1}(s_{d+2}) & f_{d+1}(s_{d+3}) & \ldots & f_{d+1}(s_{p})\\
\end{pmatrix}$$
and $f_i(x)=\prod\limits_{\mbox{\tiny$\begin{array}{c}j=1\\j\neq i\end{array}$}}^{d+1}\frac{(x-s_j)}{(s_i-s_j)},\quad i=1,2,\cdots,d+1$, and $\bar{B}'=(
    I_{\tilde q}, O)^\top$.

The block element $\bar{C}_{ij}$ of the matrix  $$\bar{C}=\begin{pmatrix}
\bar{C}_{11} & \bar{C}_{12} & \ldots & \bar{C}_{1m}\\
\vdots & \vdots & \vdots & \vdots\\
\bar{C}_{n1} & \bar{C}_{n2} & \ldots & \bar{C}_{nm}\\
    \end{pmatrix}$$
takes the form
\begin{equation*}
\bar{C}_{ij}=
    \begin{cases}
   \bar{C}_{ij}^{(1)}  & j\le d+1-\tilde q\\
   \bar{C}_{ij}^{(2)}
    & j> d+1-\tilde q
    \end{cases}
\end{equation*}
where $\bar{C}_{i,j}^{(1)}=\rev{E_{\tilde q+j,\tilde p+i}}$ is the fundamental matrix with the element at \rev{$(\tilde q+j,\tilde p+i)$} being $1$ and the others being zero, and $\rev{\bar C_{ij}^{(2)}}$ is in the form of
$$\begin{matrix}
   \begin{pmatrix}
     0 & 0 & \ldots & g_1(\rev{t_{\tilde q+j}}) & \ldots & 0\\
     0 & 0 & \ldots & g_2(\rev{t_{\tilde q+j}}) & \ldots & 0\\
     \vdots & \vdots & \vdots & \vdots & \vdots & \vdots\\
     0 & 0 & \ldots & g_{d+1}(\rev{t_{\tilde q+j}}) & \ldots & 0\\
 \end{pmatrix}\\
     \begin{matrix}
        \hphantom{\vdots}&
        \hphantom{\vdots} &
        \hphantom{\vdots} &
    \uparrow &
    \hphantom{\vdots} & \\
    \hphantom{\vdots}&
        \hphantom{\vdots} &
        \hphantom{\vdots} &
    \text{($\tilde p+i$)-th column} &
    \hphantom{\vdots} &
        \end{matrix}
\end{matrix}$$
where $g_i(x)=\prod\limits_{\mbox{\tiny$\begin{array}{c}j=1\\j\neq i\end{array}$}}^{d+1}\frac{(x-t_j)}{(t_i-t_j)},\quad i=1,2,\cdots,d+1$.

Since the first \rev{$\tilde p$} columns in each $\bar{A}'$ \rev{provide an identity block} and the first $\tilde q$ rows \rev{of each $\bar B'$ }also form an identity matrix, one can use these identity matrices to eliminate other non-zero elements in the same columns or rows. Finally, after exchanging some rows and columns, one can get the matrix
$$\bar{M}'=\begin{pmatrix}
    I_s & O\\
    O & \tilde M
\end{pmatrix}$$
where $s=\tilde p m+\tilde q n$.
The matrix $\tilde M$ can be written as
$$\tilde{M}=\begin{pmatrix}
    \Tilde{A} \\ \Tilde{C}
\end{pmatrix},$$
where $\Tilde{A}=\mathrm{diag}\underbrace{(\tilde A',\ldots,\tilde A')}_{m}\in\mathbb{R}^{m(d+1-\tilde p)\times mn}$ is a quasi-diagonal matrix\linebreak[4] with $\tilde A'=\left(
    I_{d+1-\tilde p},  S_{\tilde A}
\right)$
and $$S_{\tilde A}=\begin{pmatrix}
    f_{\tilde p+1}(s_{d+2}) & f_{\tilde p+1}(s_{d+3}) & \ldots & f_{\tilde p+1}(s_{p})\\
    \vdots & \vdots & \vdots & \vdots\\
    f_{d+1}(s_{d+2}) & f_{d+1}(s_{d+3}) & \ldots & f_{d+1}(s_{p})
\end{pmatrix}.$$

 The matrix $\tilde{C}$ is an $n\times m$ block matrix in the form
$$\tilde{C}=\begin{pmatrix}
\tilde{C}_{11} & \tilde{C}_{12} & \ldots & \tilde{C}_{1m}\\
\vdots & \vdots & \vdots & \vdots\\
\tilde{C}_{n1} & \tilde{C}_{n2} & \ldots & \tilde{C}_{nm}\\
    \end{pmatrix},$$
where $\tilde{C}_{ij}\in\rev{\mathbb{R}^{(d+1-\tilde q)\times n}}$ and
\begin{equation*}
\tilde{C}_{ij}=
    \begin{cases}
  \tilde{C}_{ij}^{(1)},  & j\le d+1-\tilde q\\
  \tilde{C}_{ij}^{(2)}, & j> d+1-\tilde q
    \end{cases}
\end{equation*}
Here, $\tilde{C}_{ij}^{(1)}=E_{ji}$ is the fundamental matrix with the $(j,i)$ element being $1$ and the rest elements being $0$, and $\tilde{C}_{ij}^{(2)}$ is in the form of
$$\begin{matrix}
    \begin{pmatrix}
     0 & 0 & \ldots & g_{\tilde q+1}(\rev{t_{\tilde q+j}}) & \ldots & 0 & 0\\
     0 & 0 & \ldots & g_{\tilde q+2}(\rev{t_{\tilde q+j}}) & \ldots & 0 & 0\\
     \vdots & \vdots & \vdots & \vdots & \vdots & \vdots & \vdots\\
     0 & 0 & \ldots & g_{d+1}(\rev{t_{\tilde q+j}}) & \ldots & 0 & 0\\
 \end{pmatrix}\\
\uparrow && \\
\text{$i$-th column} &&
 \end{matrix}.$$

In order to determine the rank of $\tilde M$, we perform elementary transformations on $\tilde M$. Write $\tilde M=\begin{pmatrix}
    \tilde M_1 & \tilde M_2 & \ldots & \tilde M_m
\end{pmatrix}$,
where $\tilde M_i=(O,(\tilde{A}^\prime)^{\top}, *)^{\top}\in\mathbb{R}^{(m(d+1-\tilde p)+n(d+1-\tilde q))\times n}$ for $i=1,2,\ldots,m$. For each ${\tilde M}_i$, we use the first $d+1-\tilde p$ columns of $\tilde A'$ (i.e., the identity matrix) to perform row operations to reduce the other non-zero elements in the same columns and obtain the reduced matrix (we still use $\tilde M_i$ to denote it)
$$\tilde M_i=\begin{pmatrix}
    O_{((i-1)(d+1-\tilde p))\times((d+1-\tilde p))} & O_{((i-1)(d+1-\tilde p))\times(p-d-1)}\\
    I_{(d+1-\tilde p)\times(d+1-\tilde p)}& S_{\tilde A}\\
    O_{((m-i)(d+1-\tilde p))\times(d+1-\tilde p)} & O_{((m-i)(d+1-\tilde p))\times(p-d-1)}\\
    \rev{O_{n(d+1-\tilde q)\times(d+1-\tilde p)}} & N_i
\end{pmatrix},$$
with $N_i\in\mathbb{R}^{(n(d+1-\tilde q))\times(p-d-1)}$.
The structure of the matrix $N_i$ is as follows. Write $N_i=\begin{pmatrix}
         N_{i1}^{\top} & N_{i2}^{\top} & \ldots & N_{in}^{\top}
     \end{pmatrix}^{\top}$, where $N_{ij}\in\mathbb{R}^{(d+1-\tilde q)\times(p-d-1)}$, $j=1,2,\ldots,n$ takes the form
     \begin{equation*}
         N_{ij}=
         \begin{cases}
         N_{ij}^{(1)}, & 1\le i\le d+1-\tilde q, \quad \rev{j\le d+1-\tilde p}\\
         N_{ij}^{(2)},
      & i> d+1-\tilde q, \quad \rev{j\le d+1-\tilde p}\\
       N_{ij}^{(3)}, & 1\le i\le d+1-\tilde q,\quad \rev{j>d+1-\tilde p}\\
       N_{ij}^{(4)}, & i>d+1-\tilde q,\quad \rev{j>d+1-\tilde p}

         \end{cases}
    \end{equation*}
    where $$N_{ij}^{(1)}=\begin{matrix}
    \begin{pmatrix}
         0 & 0 & \ldots & 0\\
         0 & 0 & \ldots & 0\\
         \vdots & \vdots & \vdots & \vdots\\
         -f_{{\tilde p}+j}(s_{d+2}) & -f_{{\tilde p}+j}(s_{d+3}) & \ldots & -f_{{\tilde p}+j}(s_{p})\\
         \vdots & \vdots & \vdots & \vdots\\
         0 & 0 & \ldots & 0\\
         0 & 0 & \ldots & 0
     \end{pmatrix} & \leftarrow & \text{the $i$-th row}
     \end{matrix}$$
     $$N_{ij}^{(2)}=\begin{pmatrix}
         -g_{\tilde q+1}(\rev{t_{\tilde q+i}})f_{\tilde p+j}(s_{d+2}) & \ldots & -g_{\tilde q+1}(\rev{t_{\tilde q+i}})f_{\tilde p+j}(s_{p})\\
         \vdots &  \vdots & \vdots\\
         -g_{d+1}(\rev{t_{\tilde q+i}})f_{\tilde p+j}(s_{d+2}) &  \ldots & -g_{d+1}(\rev{t_{\tilde q+i}})f_{\tilde p+j}(s_{p})
     \end{pmatrix}$$
     $N_{ij}^{(3)}=\rev{E_{i,j+\tilde p-d-1}}$ is the fundamental matrix, and $N_{ij}^{(4)}$ is in the form of
     $$\begin{matrix}
    \begin{pmatrix}
     0 & 0 & \ldots & g_{\tilde q+1}(\rev{t_{\tilde q+i}}) & \ldots & 0\\
     0 & 0 & \ldots & g_{\tilde q+2}(\rev{t_{\tilde q+i}}) & \ldots & 0\\
     \vdots & \vdots & \vdots & \vdots & \vdots & \vdots\\
     0 & 0 & \ldots & g_{d+1}(\rev{t_{\tilde q+i}}) & \ldots & 0\\
 \end{pmatrix}\\
\uparrow && \\
\text{the $(j+\tilde p-d-1)$-th column} &&
 \end{matrix}.$$

Now we continue to perform elementary transformations to simplify $\tilde M$. First, for each $\tilde M_i$, we apply the order $d+1-\tilde p$ identity matrix in $\tilde M_i$ to eliminate $S_{\tilde A}$ by column elementary transformations, which does not influence the matrix $N_i$.
Obviously, we have $$\rank(\tilde M_i)=d+1-\tilde p+\rank(N_i).$$

Next, we continue to simplify $N_i$.
Notice that
$N_l$, $l>d+1-\tilde q$ can be expressed as a linear combination of $N_1,N_2,\ldots,N_{d+1-\tilde q}$
$$N_l=\sum\limits_{k=1}^{d+1-\tilde q}g_{k+\tilde q}(\rev{t_{\tilde q+l}})N_k,$$
 then one can apply $N_1,N_2,\ldots,N_{d+1-\tilde q}$ (which involve $N_{ij}^{(1)}$ and $N_{ij}^{(3)}$ only) to eliminate $N_l$ for $l>d+1-\tilde q$ (which involve
 $N_{ij}^{(2)}$ and $N_{ij}^{(4)}$ only) by performing column elementary transformations.  Let $\tilde N=(N_1,N_2,\ldots, N_{d+1-\tilde q})$, then
$$
\rank(\tilde M)=m(d+1-\tilde p)+\rank(\tilde N).$$

Finally, we need to simplify
$\tilde N=\begin{pmatrix} \tilde N_1 \\ \tilde N_3\end{pmatrix}$, where
$\tilde N_1=\left(N_{ij}^{(1)}\right)$
($1\le i\le d+1-\tilde q$, $\rev{1\le j\le d+1-\tilde p}$) and $\tilde N_3=\left(N_{ij}^{(3)}\right)$ ($ 1\le i\le d+1-\tilde q$, $\rev{d+2-\tilde p\le j\le n}$). \rev{Here $j$ indexes block rows and $i$ indexes block columns; empty trailing blocks are allowed.}

From $$N_{ij}^{(1)}=\sum\limits_{\rev{k=d+2-\tilde p}}^{n}-f_{\tilde p+j}(\rev{s_{k+\tilde p}})N_{ik}^{(3)},$$
$\tilde N_1$ can be eliminated with $\tilde N_3$ through row elementary transformations. At last, the matrix $\tilde M$ is reduced to $\mathrm{diag}(I_r,O)$, and
$$\rank(\tilde M)=m(d+1-\tilde p)+\rank(\tilde N_3)=r=m(d+1-\tilde p)+(d+1-\tilde q)(p-d-1).$$
 Notice that $m=q-\tilde q$ and $n=p-\tilde p$, then $\rank(\tilde M)$ can be simplified as
$$\rank(\tilde M)=mn-(d+1-p)(d+1-q).$$
Since there are $mn$ multi-vertices in the tensor product T-connected component $L(\mathscr{T})$, we have
\begin{align*}
\dim\textsf{CVS}[L(\mathscr{T})] & =mn-\rank(\tilde M) \\
& =(p-d-1)(q-d-1).
\end{align*}
The proof is complete. \hfill$\Box$
\end{proof}

It should be emphasized that Theorem~\ref{thm3.1} is a completely new result. Prior work, such as~\cite{dim2014}, only provided upper and lower bounds on the dimension of the spline space. While these bounds enabled dimension calculations for specific cases, such as~\cite[Example~5.2]{dim2014}, they did not yield a general dimension formula. Theorem~\ref{thm3.1} offers the first systematic derivation of such a formula. The proof demonstrates that, even for the simplest tensor product T-connected components, the \rev{smoothing} cofactor method encounters significant challenges in addressing non-diagonalizable cases. Consequently, Theorem~\ref{thm3.1} indicates that the complexity of determining the dimension of the spline space primarily arises from the non-diagonalizable part of the T-connected component.

In the following, we concentrate on splines over T-meshes that have a hierarchical structure, and the tensor product T-connected component will be considered as a type of subdivision mode, referred to as the \textbf{tensor product subdivision}. A tensor product subdivision is called an \textbf{$(m,n)$-tensor product subdivision} if each cell of an $m\times n$ tensor product submesh is subdivided. The tensor product subdivision can be considered as a generalization of several other proposed subdivision modes, including cross subdivision~\cite{jin2013}, $(m,n)$-subdivision~\cite{wu2012}, and $m\times n$ division~\cite{zeng2015}. We demonstrate that under the non-degenerate tensor product subdivision mode, the dimension of the conformality vector space associated with the T-connected component of the hierarchical T-meshes can be computed by summing up the dimension of the conformality vector space associated with the T-connected component in each level. The term ``non-degenerate'' means that none of the $l$-edges in the tensor product subdivision are vanished edges. We assume that the subdivision is non-degenerate throughout the paper.

\begin{corollary}\label{cor3.3}
For a hierarchical T-mesh $\mathscr T$ under $(d-1)\times(d-1)$ tensor product subdivisions that are mutually disjoint in each level $i$, the dimension of the conformality vector space can be calculated in a recursive manner as follows:
\[
\dim\textsf{CVS}[N_{i-1}]=\dim\textsf{CVS}[T_i]+\dim\textsf{CVS}[N_i],
\]
where $N_i$ and $T_i$ are defined in Subsection~2.1.
\end{corollary}
\begin{proof}
Since the $(d-1)\times(d-1)$ tensor product subdivisions are mutually disjoint in each level, they can be handled independently. To simplify the proof, let us assume that there is only one such tensor product subdivision in each level. We prove the case for $\lev(\mathscr T)=2$ and $i=1$, as the general case for higher levels can be established through a recursive argument.

Let $N_1'$ be the subdivision in level~2 that subdivides all $4(d-1)^2$ cells in level~1. Hence, $N_1\subset N_1'$. Let $\mathscr T'$ be the T-mesh that includes $T_1$ and $N_1'$ as its T-connected component. If we can prove that $\dim\textsf{CVS}[L(\mathscr T')]=\dim\textsf{CVS}[T_1]+\dim\textsf{CVS}[N_1']$, the conclusion follows from Lemma~\ref{lem2.1}.

Since $L(\mathscr T')$ forms a tensor product T-connected component with $4d-3$ vertices on each T $l$-edge, it is completely non-diagonalizable. Therefore, by Theorem~\ref{thm3.1},
\[
\dim\textsf{CVS}[L(\mathscr T')]=[4d-3-(d+1)][4d-3-(d+1)]=(3d-4)^2.
\]
On the other hand, all the T $l$-edges in $T_1$ form a tensor product T-connected component with $2d-1$ vertices on each T $l$-edge, and it is also completely non-diagonalizable. By Theorem~\ref{thm3.1},
\[
\dim\textsf{CVS}[T_1]=[2d-1-(d+1)][2d-1-(d+1)]=(d-2)^2.
\]
All the T $l$-edges in $N_1'$ form a tensor product T-connected component with $4d-3$ vertices on each T $l$-edge. Since the number of mono-vertices on each T $l$-edge is greater than or equal to $d+1$, all the T $l$-edges in $N_1'$ have a reasonable order. Hence, $\dim\textsf{CVS}[N_1']=v_{N_1'}-(d+1)t_{N_1'}$, where $v_{N_1'}$ is the number of vertices in $N_1'$ and $t_{N_1'}$ is the number of all the T $l$-edges in $N_1'$. Since $v_{N_1'}=4(d-1)(3d-2)$ and $t_{N_1'}=4(d-1)$,
\[
\dim\textsf{CVS}[N_1']=4(d-1)(3d-2)-4(d-1)(d+1)=4(d-1)(2d-3).
\]
Therefore,
\begin{align*}
\dim\textsf{CVS}[L(\mathscr T')]&=(3d-4)^2=(d-2)^2+4(d-1)(2d-3)\\
&=\dim\textsf{CVS}[T_1]+\dim\textsf{CVS}[N_1'].
\end{align*}
By Lemma~\ref{lem2.1}, the conclusion holds in the general case. \hfill$\Box$
\end{proof}

Next, we extend the result in Corollary~\ref{cor3.3} to general tensor product subdivisions. Before presenting the central theorem, we make necessary preparations.

\begin{lemma}[See~\cite{zeng2015}]\label{lem3.2}\label{lem5.2}
Given a T-mesh $\mathscr T$, suppose that $L(\mathscr T)$ is the T-connected component of $\mathscr T$ and $\{L_1,L_2\}$ is a bipartite partition of $L(\mathscr T)$ with $L_1\subset L(\mathscr T)$. Then
\[
\dim\textsf{CVS}[L(\mathscr T)]\le\dim\textsf{CVS}[L_1]+\dim\textsf{CVS}[L_2].
\]
\end{lemma}

The following lemma illustrates a property of the spline space with homogeneous boundary conditions.
\begin{lemma}[See~\cite{zeng2015}]\label{lem3.3}
Given a T-mesh $\mathscr T$, let $\mathscr T'$ be the tensor-product mesh formed by all the boundary $l$-edges and cross-cuts of $\mathscr T$, and $L$ be the set of all the rays and T $l$-edges of $\mathscr T$. If $\mathscr T'$ has at least $d+1$ horizontal $l$-edges and $d+1$ vertical $l$-edges, then $\dim\bar S_d(\mathscr T)=\dim\bar S_d(\mathscr T')+\dim\textsf{CVS}[L]$.
\end{lemma}

Now the main theorem and its proof are outlined below, and the methodology used is comparable to the one presented in~\cite{zeng2015}.

\begin{theorem}\label{thm3.2}
Given a hierarchical T-mesh $\mathscr T$ under the tensor product subdivision mode, and consider the spline space $S_d(\mathscr T)$. If there is only one tensor product T-connected component in level~1 and there are at least $d+1$ horizontal and $d+1$ vertical $l$-edges in level~0, then we have the following equality concerning the dimension of the conformality vector space:
\[
\dim\textsf{CVS}[N_0]=\dim\textsf{CVS}[T_1]+\dim\textsf{CVS}[N_1].
\]
\end{theorem}
\begin{proof}
If all the T-$l$ edges in level~1 of $L(\mathscr T)$ are in a reasonable order, by Proposition~\ref{prop2.2}, the conclusion holds. Otherwise, consider the mesh $\mathscr T'$ generated by $N_0$ and the edges of the cells in level~0 subdivided by $T_1$ (an illustration is shown in Figure~\ref{fig9}).
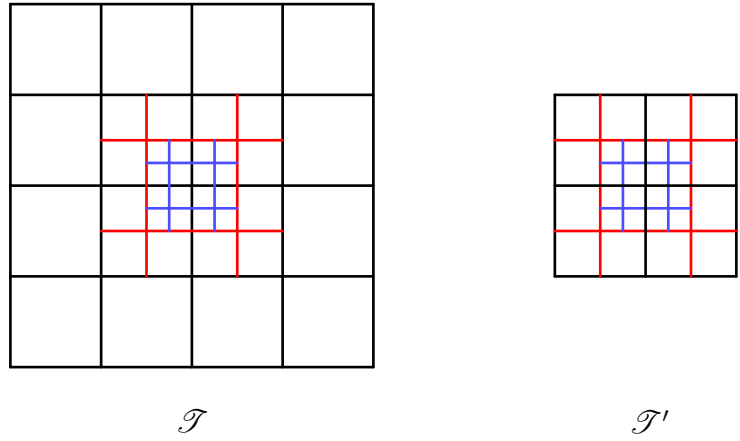
\begin{figure}[htbp]
    \centering
    \definecolor{ududff}{rgb}{0.30196078431372547,0.30196078431372547,1.}
\definecolor{ffqqqq}{rgb}{1.,0.,0.}
\begin{tikzpicture}[line cap=round,line join=round,>=triangle 45,x=1.0cm,y=1.0cm,scale=1.2]
\draw [line width=1.pt] (1.,1.)-- (5.,1.);
\draw [line width=1.pt] (1.,1.)-- (1.,5.);
\draw [line width=1.pt] (1.,5.)-- (5.,5.);
\draw [line width=1.pt] (5.,5.)-- (5.,1.);
\draw [line width=1.pt] (1.,4.)-- (5.,4.);
\draw [line width=1.pt] (1.,3.)-- (5.,3.);
\draw [line width=1.pt] (1.,2.)-- (5.,2.);
\draw [line width=1.pt] (2.,1.)-- (2.,5.);
\draw [line width=1.pt] (3.,5.)-- (3.,1.);
\draw [line width=1.pt] (4.,1.)-- (4.,5.);
\draw [line width=1.pt,color=ffqqqq] (2.,3.5)-- (4.,3.5);
\draw [line width=1.pt,color=ffqqqq] (2.,2.5)-- (4.,2.5);
\draw [line width=1.pt,color=ffqqqq] (2.5,2.)-- (2.5,4.);
\draw [line width=1.pt,color=ffqqqq] (3.5,4.)-- (3.5,2.);
\draw [line width=1.pt,color=ududff] (2.5,3.25)-- (3.5,3.25);
\draw [line width=1.pt,color=ududff] (3.5,2.75)-- (2.5,2.75);
\draw [line width=1.pt,color=ududff] (3.25,3.5)-- (3.25,2.5);
\draw [line width=1.pt,color=ududff] (2.75,3.5)-- (2.75,2.5);
\draw [line width=1.pt] (7.,4.)-- (9.,4.);
\draw [line width=1.pt] (7.,4.)-- (7.,2.);
\draw [line width=1.pt] (7.,2.)-- (9.,2.);
\draw [line width=1.pt] (9.,2.)-- (9.,4.);
\draw [line width=1.pt,color=ffqqqq] (7.,3.5)-- (9.,3.5);
\draw [line width=1.pt,color=ffqqqq] (7.,2.5)-- (9.,2.5);
\draw [line width=1.pt,color=ffqqqq] (7.5,4.)-- (7.5,2.);
\draw [line width=1.pt,color=ffqqqq] (8.5,2.)-- (8.5,4.);
\draw [line width=1.pt,color=ududff] (7.5,3.25)-- (8.5,3.25);
\draw [line width=1.pt,color=ududff] (7.5,2.75)-- (8.5,2.75);
\draw [line width=1.pt,color=ududff] (7.75,2.5)-- (7.75,3.5);
\draw [line width=1.pt,color=ududff] (8.25,3.5)-- (8.25,2.5);
\draw [line width=1.pt] (7.,3.)-- (9.,3.);
\draw [line width=1.pt] (8.,2.)-- (8.,4.);
\draw (2.7,0.65) node[anchor=north west] {$\mathscr{T}$};
\draw (7.7,0.65) node[anchor=north west] {$\mathscr{T}'$};
\end{tikzpicture}
    \caption{\label{fig9}(Color online) A T-mesh $\mathscr T$ and corresponding $\mathscr T^{\prime}$ in Theorem~\ref{thm3.2}}
\end{figure}

Since $T_1$ is a tensor product T-connected component, by Lemma~\ref{lem3.3}, we have
\begin{equation}\label{eq3.5}
\dim\bar S_{\rev{d}}(\mathscr T')=\dim\bar S_{\rev{d}}(\mathscr T'_{TP})+\dim\textsf{CVS}[N_1],
\end{equation}
where $\mathscr T'_{TP}$ is the tensor-product mesh generated by $T_1$ and \rev{the level-0 grid lines in $\mathscr T'$, including its boundary}.

Suppose that each horizontal $l$-edge has $p$ vertices and each vertical $l$-edge has $q$ vertices in $T_1$. Since there are $p$ \rev{vertical} $l$-edges and $q$ \rev{horizontal} $l$-edges in $\mathscr T'_{TP}$, we can easily calculate that
\[
\dim\bar S_{\rev{d}}(\mathscr T'_{TP})=[p-(d+1)]\times[q-(d+1)]=(d+1-p)(d+1-q).
\]
Moreover, as $T_1$ is a non-diagonalizable tensor product T-connected component, by Theorem~\ref{thm3.1},
\[
\dim\textsf{CVS}[T_1]=(d+1-p)(d+1-q).
\]
\rev{Comparing these two dimension formulas}, we have
\begin{equation}\label{eq 3.6}
\dim\rev{\bar S_d(\mathscr T'_{TP})}=\dim\textsf{CVS}[T_1].
\end{equation}
Since $\dim\bar S_{\rev{d}}(\mathscr T')\le\dim\textsf{CVS}[N_0]$, combining this with~\eqref{eq3.5} and~\eqref{eq 3.6}, we get
\[
\dim\textsf{CVS}[T_1]+\dim\textsf{CVS}[N_1]\le\dim\textsf{CVS}[N_0].
\]
On the other hand, since $\{N_1,T_1\}$ is a bipartite partition of $N_0$, by Lemma~\ref{lem3.2},
\[
\dim\textsf{CVS}[N_0]\le\dim\textsf{CVS}[T_1]+\dim\textsf{CVS}[N_1].
\]
Thus $\dim\textsf{CVS}[N_0]=\dim\textsf{CVS}[T_1]+\dim\textsf{CVS}[N_1]$ follows immediately. \hfill$\Box$
\end{proof}

{\color{blue}
\begin{remark}
In the preceding proof, $\mathscr T'_{TP}$ is the tensor-product mesh after the first refinement inside the auxiliary rectangle, with the level-0 grid lines retained. Thus $p$ counts vertical lines and $q$ counts horizontal lines. If a later refinement produces additional cross-cuts, those lines belong to $N_1$, not to the prescribed core. The cofactor decomposition used in Lemma~\ref{lem3.3} also applies to this core: its $(d+1)^2$ row dependencies are obtained from polynomials in $P_{(d,d)}$, and extend to the added lines by assigning opposite horizontal and vertical contributions at each vertex. Therefore the equations on $N_1$ give precisely the compatibility conditions for solving the equations on the core.

For the inequality $\dim\bar S_d(\mathscr T')\le\dim\textsf{CVS}[N_0]$, extend a spline in $\bar S_d(\mathscr T')$ by zero and take its vertex cofactors on $N_0$. This defines a linear map into $\textsf{CVS}[N_0]$. If its image is zero, the jumps across all refinement edges vanish, so the spline belongs to the homogeneous spline space on the unrefined parent rectangle. Write $m,n$ for the numbers of horizontal and vertical T $l$-edges of $T_1$. In the non-diagonalizable case, the parent rectangle has $p-n<d+1$ vertical and $q-m<d+1$ horizontal grid lines. Its homogeneous spline space is zero. The cofactor map is consequently injective, which gives the asserted inequality. The black level-0 lines in Figure~\ref{fig9} are retained in the core.
\end{remark}
}

\section{Dimension of the spline space $S_d(\mathscr T)$ over a hierarchical T-mesh}
Theorem~\ref{thm3.2} presents a general conclusion. As an application, we use the result to derive a dimension formula of bi-degree $(d,d)$ spline spaces with the highest order of smoothness over a certain type of hierarchical T-mesh. This result can be regarded as an extension of the result in~\cite{zeng2015}. It is worth noting that a similar outcome can be obtained through the method in~\cite{wu2012}. However, the proof here is more concise, and the result is more comprehensive.

\begin{theorem}\label{thm4.1}
Given a tensor product mesh $\mathscr T_0$ that contains at least $(d+1)$ horizontal $l$-edges and $(d+1)$ vertical $l$-edges, let $\mathscr T$ be the T-mesh created by a collection of $(d-1)\times(d-1)$ tensor product subdivisions, where the subdivision regions may partially overlap. \rev{Here only interior cells are subdivided.} Then, the dimension of the conformality vector space of $L(\mathscr T)$ can be determined as $\dim\textsf{CVS}[L(\mathscr T)]=v-(d+1)t+\gamma$, where $v$ represents the number of vertices on $L(\mathscr T)$, $t$ is the number of T $l$-edges in $L(\mathscr T)$, and $\gamma$ is the number of isolated $(d-1)\times(d-1)$ tensor product subdivisions.
\end{theorem}

\begin{proof}
Let $\{L_1,L_2\}$ be a bipartite partition of $L(\mathscr T)$, where $N(l)=d-1$ for any $l\in L_1$ and $N(l)>d-1$ for any $l\in L_2$, where $N(l)$ represents the number of cells crossed by the edge $l$. There are the following three cases to consider:
\begin{enumerate}
\item[(1)] $L_2=\varnothing$.

In this case, all the $(d-1)\times(d-1)$ submeshes in the subdivision are separated from each other, as $N(l)=d-1$ for each $l\in L(\mathscr T)$. Thus, the problem reduces to considering each isolated connected component separately in this case. For each isolated connected component $T'$, it is a non-diagonalizable tensor product T-connected component. Thus, according to Theorem~\ref{thm3.1}, the dimension of the conformality vector space can be calculated as $\dim\textsf{CVS}[T']=(d-2)^2$. Notice that the number of vertices in $T'$ is $v'=(d-1)(3d-1)$ and the number of T $l$-edges is $t'=2(d-1)$. Using this information, the dimension formula can be expressed as $\dim\textsf{CVS}[T']=v'-(d+1)t'+1$. Then, for all $\gamma$ isolated connected components, the dimension formula is $\dim\textsf{CVS}[L(\mathscr T)]=v-(d+1)t+\gamma$, where $v=\gamma v'$ is the number of vertices on $L(\mathscr T)$ and $t=\gamma t'$ is the number of T $l$-edges of $L(\mathscr T)$. Thus, the dimension formula is calculated as required.

\item[(2)] $L_2\ne\varnothing$ and there are no isolated $(d-1)\times(d-1)$ subdivisions.

We now prove that all the T $l$-edges in $L(\mathscr T)$ have a reasonable order. Let $L_1=l_1,l_2,\ldots,l_m$ and $L_2=l_{m+1},\ldots,l_{m+n}$. Without loss of generality, assume that $l_1,\ldots,l_{\tilde m}$ are the edges in $L_1$ that do not intersect with edges in $L_2$, where $\rev{0}\le\tilde m\le m$. We claim that $l_1\succ l_2\succ\cdots\succ l_{\tilde m}\succ\cdots\succ l_m\succ\cdots\succ l_{m+n}$ is a reasonable order of $L(\mathscr T)$.

First, we prove that $l_i$ does not intersect with $l_j$ for $1\le i,j\le\tilde m$ and $i\ne j$. In fact, if $l_i$ intersects with $l_j$, then one of them must be a horizontal $l$-edge and the other a vertical $l$-edge. Consider the $(d-1)\times(d-1)$ submesh $\mathscr T'$ that contains both $l_i$ and $l_j$. Since $N(l_i)=N(l_j)=d-1$ and $l_i$ and $l_j$ have no intersection with the $l$-edges in $L_2$, all of the \rev{$2(d-1)$} T $l$-edges in $\mathscr T'$ do not intersect with the $l$-edges in $L_2$. Therefore, $\mathscr T'$ is an isolated submesh, which contradicts the hypothesis. Hence, $\rev{r(l_k)}=n(l_k)\rev{\ge}d+1$ for $k=1,2,\ldots,\tilde m$, where $n(l_k)$ is the number of vertices on $l_k$.

Second, consider the edges $\tilde l_{\tilde m+1},\ldots,\tilde l_m$. These edges are obtained from $l_r$ by removing the intersections with \rev{$l_1,\ldots,l_{r-1}$}, but retaining the intersections of the $l$-edges in $L_2$ with $\tilde l_r$. Since $N(\tilde l_r)=d-1$, \rev{the $d$ intersections with level-0 lines and at least one intersection with $L_2$ remain, so $r(l_r)\ge d+1$.}

Finally, for $\tilde l_p\in L_2$, $p=m+1,\ldots,m+n$, since $N(\tilde l_p)>d-1$, then \rev{at least $d+1$ intersections with level-0 lines remain, so $r(l_p)\ge d+1$.}

\rev{In short, $r(l_q)\ge d+1$ for $q=1,\ldots,m+n$. Thus} $l_1\succ l_2\succ\cdots\succ l_{\tilde m}\succ\cdots\succ l_m\succ\cdots\succ l_{m+n}$ is a reasonable order of $L(\mathscr T)$. Therefore, the dimension formula in this case is $\dim\textsf{CVS}[L(\mathscr T)]=v-(d+1)t$. Since $\gamma=0$ in this case, the dimension formula is still correct.

\item[(3)] $L_2\ne\varnothing$ and there exist isolated $(d-1)\times(d-1)$ subdivisions.

The general case is just the combination of the above two cases, and thus the assertion follows easily.\hfill$\Box$
\end{enumerate}
\end{proof}

Notice that Theorems~\ref{thm3.2} and~\ref{thm4.1} can be generalized to the case of the spline space with homogeneous conditions. For spline spaces with homogeneous boundary conditions, all $l$-edges of $\mathscr T$ meet the global conformality conditions. As a result, the biggest difference between $S_d(\mathscr T)$ with $\bar S_d(\mathscr T)$ in Theorems~\ref{thm3.2} and~\ref{thm4.1} is that $\bar S_d(\mathscr T)$ allows some boundary cells to be subdivided but $S_d(\mathscr T)$ cannot.

\begin{theorem}\label{thm4.2}
Given a hierarchical T-mesh $\mathscr T$ created by subdividing a collection of $(d-1)\times(d-1)$ tensor product submeshes at each level, where the subdivision regions may partially overlap, then
\begin{equation}\label{dimht}
\dim S_d(\mathscr T)=v+db+\rev{(d+1)^2}+\gamma_0-t(d+1),
\end{equation}
where $v$ is the number of \rev{distinct} vertices in $\mathscr T$, $t$ is the number of \rev{maximal $l$-edges, including the four boundary $l$-edges,} in $\mathscr T$, $b$ is the number of \rev{distinct boundary vertices, including the four corners (equivalently, the number of boundary small edges),} and $\gamma_0$ is the number of isolated connected components that subdivide $(d-1)\times(d-1)$ interior cells in all \rev{levels}.
\end{theorem}

\begin{proof}
Consider the spline space $\bar S_d(\mathscr T)$. \rev{For this intermediate homogeneous-space calculation, take the initial mesh to have at least $d+1$ grid lines in each direction; this condition holds automatically for the extended initial mesh used below.} In each level $i$, $i=1,2,\ldots,\lev(\mathscr T)$, if a connected component is an isolated $(d-1)\times(d-1)$ tensor product component, the dimension of the conformality vector space can be calculated recursively using Corollary~\ref{cor3.3}. Additionally, if a connected component is not an isolated $(d-1)\times(d-1)$ tensor product subdivision, the dimension of the conformality vector space can still be calculated in a recursive manner, as stated in Theorem~\ref{thm4.1}. This is because if a connected component is not isolated, then it has a reasonable order, as indicated by~\eqref{rea}. Thus, we have $\dim\bar S_d(\mathscr T)=\dim\bar S_d(\mathscr T_0)+\dim\textsf{CVS}[L]$, where $\mathscr T_0$ is the tensor-product mesh in level~0, and $L$ is the set of all the \rev{positive-level $l$-edges with their positive-level vertex variables} in $\mathscr T$.

Since $\dim\textsf{CVS}[L]$ can be calculated recursively, we have
\[
\dim\textsf{CVS}[L]=\sum_{i=1}^{\lev(\mathscr T)}\dim\textsf{CVS}[L_i]=\sum_{i=1}^{\lev(\mathscr T)}(v_i-(d+1)t_i+\gamma_i),
\]
where $L_i$ is the set of all the $l$-edges \rev{and vertex variables born} in level $i$, $v_i$ is the number of \rev{distinct vertices born} in level $i$, $t_i$ is the number of $l$-edges in level $i$, and $\gamma_i$ is the number of connected components that are isolated $(d-1)\times(d-1)$ tensor product subdivisions in level $i-1$, $i=1,2,\ldots,\lev(\mathscr T)$.

It is straightforward to verify that $\dim\bar S_d(\mathscr T_0)=v_0-(d+1)t_0+(d+1)^2$. Therefore,
\[
\dim\bar S_d(\mathscr T)=v_0-(d+1)t_0+(d+1)^2+\sum_{i=1}^{\lev(\mathscr T)}(v_i-(d+1)t_i+\gamma_i).
\]
Let $v=\sum_{i=0}^{\lev(\mathscr T)}v_i$ be the number of all the vertices in $\mathscr T$, $t=\sum_{i=0}^{\lev(\mathscr T)}t_i$ be the number of all the $l$-edges in $\mathscr T$, and $\gamma=\sum_{\rev{i=1}}^{\lev(\mathscr T)}\gamma_i$. Then
\begin{equation}\label{dimf}
\dim\bar S_d(\mathscr T)=v-(d+1)t+(d+1)^2+\gamma.
\end{equation}
Next, we show that $\dim\bar S_d(\bar{\mathscr T})$ also satisfies the equation~\eqref{dimf}. Notice that there is a structural difference between the hierarchical T-mesh $\mathscr T$ and its extended T-mesh $\bar{\mathscr T}$. To concisely highlight this difference, we discuss it at each level of the two hierarchical T-meshes $\mathscr T$ and $\bar{\mathscr T}$. It should be emphasized that, given our consideration of the spline space with homogeneous boundary conditions, some boundary cells in both $\mathscr T$ and $\bar{\mathscr T}$ are permitted to be subdivided.

In a particular level $i_0$, consider the bipartite partition $\{L_1,L_2\}$ of the T-connected component of $\mathscr T$ in level $i_0$, where $L_1$ is non-diagonalizable and $L_2$ is diagonalizable. Similarly, $\{\bar L_1,\bar L_2\}$ is a bipartite partition of the T-connected component of $\bar{\mathscr T}$ in level $i_0$, with $\bar L_1$ being non-diagonalizable and $\bar L_2$ being diagonalizable. $L_1$ consists of all isolated $(d-1)\times(d-1)$ tensor product subdivisions that result from subdividing some cells of $\mathscr T$ in level $i_0-1$. To further divide $L_1$, let $L_1=\{L_{11},L_{12}\}$, where $L_{11}$ is the set of all isolated $(d-1)\times(d-1)$ tensor product subdivisions that only subdivide interior cells of $\mathscr T$ in level $i_0-1$, and $L_{12}$ is the set of all isolated $(d-1)\times(d-1)$ tensor product subdivisions that subdivide some boundary cells of $\mathscr T$ in level $i_0-1$. We demonstrate that $\bar L_1=L_{11}$.

Since the isolated $(d-1)\times(d-1)$ tensor product subdivisions that only subdivide interior cells of $\mathscr T$ in level $i_0-1$ do not require any extension in the extended T-mesh $\bar{\mathscr T}$, they remain non-diagonalizable, i.e., $L_{11}\subset\bar L_1$.

Next, we demonstrate that extending the $l$-edges in $L_{12}$ to the boundary of $\bar{\mathscr T}$ will result in the extended T-mesh $\bar{\mathscr T}$ no longer being considered a hierarchical T-mesh under cross subdivision. To illustrate this, consider the T-mesh shown on the left of Figure~\ref{figg9}. Let us denote it by $\mathscr T$. The spline space $S_2(\mathscr T)$ is considered. The red-marked region in $\mathscr T$ represents a boundary cell in level~0 that is subdivided by cross subdivision in level~1. Upon extending the $l$-edges, as shown on the right, the red region cannot be considered as cells that are subdivided through cross subdivision.

Fortunately, the \rev{extensions} of $l$-edges in $L_{12}$ to $\bar{\mathscr T}$ have a reasonable order. This is because when the $l$-edges of $L_{12}$ are extended to the boundary of $\bar{\mathscr T}$, $d$ new vertices are added to these extended $l$-edges. This ensures that the extended $l$-edges have a reasonable order, and they no longer remain non-diagonalizable. As a result, it follows that $\bar L_1=L_{11}$.

Actually, \eqref{dimf} holds when the set of all $l$-edges in $\mathscr T$ is divided into a bipartite partition $\{L_1,L_2\}$, where $L_1$ is a set of isolated $(d-1)\times(d-1)$ tensor product subdivisions, and $L_2$ is a set of $l$-edges that have a reasonable order. Since all the $l$-edges in $\bar{\mathscr T}$ can also be divided into such a bipartite partition $\{\bar L_1,\bar L_2\}$, the dimension formula~\eqref{dimf} also holds for the spline space with homogeneous boundary conditions over $\bar{\mathscr T}$. Let $\bar{\mathscr T}$ have $\bar v$ vertices, $\bar t$ $l$-edges, and $\bar\gamma=\gamma_0$ isolated $(d-1)\times(d-1)$ tensor product subdivisions that subdivide only interior cells. By~\eqref{dimf}, we have
\begin{equation}\label{dimext}
\dim\bar S_d(\bar{\mathscr T})=\bar v-(d+1)\bar t+(d+1)^2+\gamma_0.
\end{equation}
\begin{figure}[htbp]
    \centering
    \definecolor{zzttqq}{rgb}{0.6,0.2,0.}
\begin{tikzpicture}[line cap=round,line join=round,>=triangle 45,x=1.0cm,y=1.0cm,scale=1.3]
\fill[line width=1.pt,color=zzttqq,fill=zzttqq,fill opacity=0.10000000149011612] (1.,2.) -- (2.,2.) -- (2.,1.) -- (1.,1.) -- cycle;
\fill[line width=1.pt,color=zzttqq,fill=zzttqq,fill opacity=0.10000000149011612] (6.603860846631458,2.0101822119607826) -- (8.00241051747193,2.0049441982123164) -- (8.00237819971885,0.5973942736378075) -- (6.603860846631458,0.5973942736378075) -- cycle;
\draw [line width=1.pt] (1.,1.)-- (4.,1.);
\draw [line width=1.pt] (4.,1.)-- (4.,4.);
\draw [line width=1.pt] (4.,4.)-- (1.,4.);
\draw [line width=1.pt] (1.,4.)-- (1.,1.);
\draw [line width=1.pt] (1.,3.)-- (4.,3.);
\draw [line width=1.pt] (1.,2.)-- (4.,2.);
\draw [line width=1.pt] (2.,1.)-- (2.,4.);
\draw [line width=1.pt] (3.,4.)-- (3.,1.);
\draw [line width=1.pt] (1.,1.5)-- (2.,1.5);
\draw [line width=1.pt] (1.5,1.)-- (1.5,2.);
\draw [line width=1.pt] (1.5,1.75)-- (2.,1.75);
\draw [line width=1.pt] (1.75,1.5)-- (1.75,2.);
\draw [line width=1.pt] (6.6,4.4)-- (6.6,0.6);
\draw [line width=1.pt] (6.6,0.6)-- (10.4,0.6);
\draw [line width=1.pt] (10.4,0.6)-- (10.4,4.4);
\draw [line width=1.pt] (10.4,4.4)-- (6.6,4.4);
\draw [line width=1.pt] (6.6,4.2)-- (10.4,4.2);
\draw [line width=1.pt] (6.6,0.8)-- (10.4,0.8);
\draw [line width=1.pt] (10.2,4.4)-- (10.2,0.6);
\draw [line width=1.pt] (6.8,4.4)-- (6.8,0.6);
\draw [line width=1.pt] (7,4.4)-- (7,0.6);
\draw [line width=1.pt] (10,0.6)-- (10,4.4);
\draw [line width=1.pt] (6.6,4)-- (10.4,4);
\draw [line width=1.pt] (6.6,1)-- (10.4,1);
\draw [line width=1.pt] (8,4.4)-- (8,0.6);
\draw [line width=1.pt] (9,0.6)-- (9,4.4);
\draw [line width=1.pt] (6.6,3)-- (10.4,3);
\draw [line width=1.pt] (6.6,2)-- (10.4,2);
\draw [line width=1.pt] (7.5,2)-- (7.5,0.6);
\draw [line width=1.pt] (8,1.5)-- (6.6,1.5);
\draw [line width=1.pt] (7.5,1.75)-- (8,1.75);
\draw [line width=1.pt] (7.75,1.5)-- (7.75,2);
\draw [line width=1.pt,color=zzttqq] (1.,2.)-- (2.,2.);
\draw [line width=1.pt,color=zzttqq] (2.,2.)-- (2.,1.);
\draw [line width=1.pt,color=zzttqq] (2.,1.)-- (1.,1.);
\draw [line width=1.pt,color=zzttqq] (1.,1.)-- (1.,2.);
\draw [line width=2.pt,color=zzttqq] (6.603860846631458,2.0101822119607826)-- (8.00241051747193,2.0049441982123164);
\draw [line width=2.pt,color=zzttqq] (8.00241051747193,2.0049441982123164)-- (8.00237819971885,0.5973942736378075);
\draw [line width=2.pt,color=zzttqq] (8.00237819971885,0.5973942736378075)-- (6.603860846631458,0.5973942736378075);
\draw [line width=2.pt,color=zzttqq] (6.603860846631458,0.5973942736378075)-- (6.603860846631458,2.0101822119607826);
\draw (2.35,0.01230027897879689) node[anchor=north west] {$\mathscr T$};
\draw (8.35,0.01230027897879689) node[anchor=north west] {$\bar{\mathscr T}$};
\end{tikzpicture}
    \caption{\label{figg9}(Color online) Structural difference between a T-mesh and its corresponding extended T-mesh}
\end{figure}
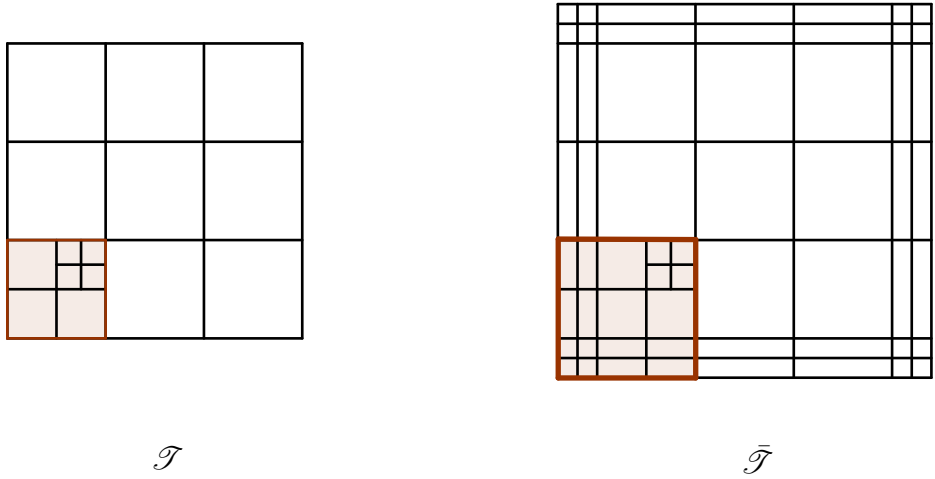

On the other hand, the number of vertices and edges in $\bar{\mathscr T}$ can be expressed as
\[
\bar v=v+4d^2+db\rev{{}+4d}\qquad\text{and}\qquad\bar t=t+4d,
\]
respectively, where $b$ is the number of boundary vertices of $\mathscr T$. \rev{Indeed, the new vertices number $d(b-4)+8d+4d^2$: every non-corner boundary vertex contributes $d$, each of the four corners contributes $2d$, and the four exterior corner regions contribute $d^2$ each.} Thus
\begin{align*}
\dim S_d(\mathscr T)&=\dim\bar S_d(\bar{\mathscr T})\\
&=\bar v-(d+1)\bar t+(d+1)^2+\bar\gamma\\
&=(v+4d^2+db\rev{{}+4d})-(d+1)(t+4d)+(d+1)^2+\bar\gamma\\
&=v+db+\rev{(d+1)^2}+\bar\gamma-t(d+1).
\end{align*}
This completes the proof of the theorem. \hfill$\Box$
\end{proof}

Notice that the dimension formula~\rev{\eqref{dimht}} is consistent with the dimension formula for \textit{$(m,n)$-subdivision T-meshes} when $m=n=d$ as described in~\cite{wu2012}. However, the T-mesh in Theorem~\ref{thm4.2} allows for some partial overlapping subdivisions, making it a more general case than the one described in~\cite{wu2012}.

Since the dimension of the spline space defined in Theorem~\ref{thm4.2} is stable, we can build up a strategy to modify an arbitrary hierarchical T-mesh $\mathscr T$ into a new hierarchical T-mesh $\mathscr T'$ based on $(d-1)\times(d-1)$ tensor product subdivision such that the dimension of the spline space $S_d(\mathscr T')$ is stable. In the following, we use an example to demonstrate the modification strategy.

\begin{example}
\rev{Vanished edges are allowed in this mesh-modification example.} Consider the spline space $S_4(\mathscr T)$ with $\mathscr T$ shown on the left of Figure~\ref{fig11}, where the dashed lines represent the vanished edges. With some effort, one can show that the dimension of the spline space $S_4(\mathscr T)$ is unstable. In the following, we modify $\mathscr T$ into a new hierarchical T-mesh $\mathscr T'$ such that the spline space $S_4(\mathscr T')$ is stable.
\end{example}

The basic idea is as follows. We modify the hierarchical T-mesh $\mathscr T$ to make it become a hierarchical T-mesh $\mathscr T'$ formed by subdividing a collection of $(d-1)\times(d-1)$ tensor product submeshes at each level, and $\mathscr T\subset\mathscr T'$. To do so, we use a $(d-1)\times(d-1)$ tensor product submesh as a template to slide along the boundary of the subdivided region until the boundary is covered by a collection of $(d-1)\times(d-1)$ tensor product submeshes. The cells covered by the submeshes are subdivided accordingly.

For the T-mesh in the Figure~\ref{fig11}, we first find a $3\times3$ tensor product submesh (whose \rev{occupied} region is denoted as $ABCD$), which covers part of the boundary of the subdivided region. Then all the cells within $ABCD$ are subdivided if they are not subdivided yet. Next, we find another $3\times3$ tensor product submesh (whose occupied region is denoted as $EFGH$), which covers the \rev{rest} of the boundary of the subdivided region, and subdivide the cells within $EFGH$ if they are not subdivided. The final outcome T-mesh $\mathscr T'$ is shown in Figure~\ref{fig12}.
\begin{figure}[htbp]
    \centering
 \definecolor{qqqqff}{rgb}{0.,0.,1.}
\definecolor{ffqqqq}{rgb}{1.,0.,0.}
\begin{tikzpicture}[line cap=round,line join=round,>=triangle 45,x=1.0cm,y=1.0cm]
\draw [line width=1.pt] (1.,1.)-- (7.,1.);
\draw [line width=1.pt] (1.,1.)-- (1.,7.);
\draw [line width=1.pt] (1.,7.)-- (7.,7.);
\draw [line width=1.pt] (7.,7.)-- (7.,1.);
\draw [line width=1.pt] (1.,6.)-- (7.,6.);
\draw [line width=1.pt] (1.,5.)-- (7.,5.);
\draw [line width=1.pt] (1.,4.)-- (7.,4.);
\draw [line width=1.pt] (1.,3.)-- (7.,3.);
\draw [line width=1.pt] (1.,2.)-- (7.,2.);
\draw [line width=1.pt] (2.,7.)-- (2.,1.);
\draw [line width=1.pt] (6.,7.)-- (6.,1.);
\draw [line width=1.pt] (3.,7.)-- (3.,1.);
\draw [line width=1.pt] (4.,1.)-- (4.,7.);
\draw [line width=1.pt] (5.,7.)-- (5.,1.);
\draw [line width=1.pt] (2.,4.5)-- (5.,4.5);
\draw [line width=1.pt,dotted] (2.5,5.)-- (2.5,4.);
\draw [line width=1.pt] (3.5,6.)-- (3.5,3.);
\draw [line width=1.pt] (4.5,5.)-- (4.5,2.);
\draw [line width=1.pt] (3.,3.5)-- (6.,3.5);
\draw [line width=1.pt,dotted] (4.,2.5)-- (5.,2.5);
\draw [line width=1.pt,dotted] (5.5,4.)-- (5.5,3.);
\draw [line width=1.pt,dotted] (3.,5.5)-- (4.,5.5);
\draw [line width=1.pt] (9.,7.)-- (9.,1.);
\draw [line width=1.pt] (9.,1.)-- (15.,1.);
\draw [line width=1.pt] (15.,1.)-- (15.,7.);
\draw [line width=1.pt] (15.,7.)-- (9.,7.);
\draw [line width=1.pt,color=ffqqqq] (10.,6.)-- (10.,3.);
\draw [line width=1.pt,color=ffqqqq] (10.,3.)-- (13.,3.);
\draw [line width=1.pt,color=ffqqqq] (13.,3.)-- (13.,6.);
\draw [line width=1.pt,color=ffqqqq] (13.,6.)-- (10.,6.);
\draw [line width=1.pt,dotted,color=ffqqqq] (10.,5.5)-- (11,5.5);
\draw [line width=1.pt,dotted,color=ffqqqq] (10.5,6.)-- (10.5,5.);
\draw [line width=1.pt,dotted,color=ffqqqq] (10.,3.5)-- (11.,3.5);
\draw [line width=1.pt,dotted,color=ffqqqq] (10.5,4)-- (10.5,3.);
\draw [line width=1.pt] (10.,4.5)-- (13.,4.5);
\draw [line width=1.pt] (10.5,5.)-- (10.5,4.);
\draw [line width=1.pt] (11.5,6.)-- (11.5,3.);
\draw [line width=1.pt] (11.,5.5)-- (12,5.5);
\draw [line width=1.pt] (11.,3.5)-- (14.,3.5);
\draw [line width=1.pt] (12.5,5.)-- (12.5,2.);
\draw [line width=1.pt] (9.,6.)-- (10.,6.);
\draw [line width=1.pt] (13.,6.)-- (15.,6.);
\draw [line width=1.pt] (9.,5.)-- (15.,5.);
\draw [line width=1.pt] (9.,4.)-- (15.,4.);
\draw [line width=1.pt] (9.,3.)-- (10.,3.);
\draw [line width=1.pt] (13.,3.)-- (15.,3.);
\draw [line width=1.pt] (9.,2.)-- (15.,2.);
\draw [line width=1.pt] (10.,7.)-- (10.,6.);
\draw [line width=1.pt] (10.,3.)-- (10.,1.);
\draw [line width=1.pt] (11.,7.)-- (11.,1.);
\draw [line width=1.pt] (12.,7.)-- (12.,1.);
\draw [line width=1.pt] (13.,7.)-- (13.,6.);
\draw [line width=1.pt] (13.,3.)-- (13.,1.);
\draw [line width=1.pt] (14.,7.)-- (14.,1.);
\draw [line width=1.pt,dotted] (12.,2.5)-- (13.,2.5);
\draw [line width=1.pt,dotted] (13.5,4.)-- (13.5,3.);
\draw [line width=1.pt,dotted,color=ffqqqq] (12.,5.5)-- (13.,5.5);
\draw [line width=1.pt,dotted,color=ffqqqq] (12.5,5)-- (12.5,6.);
\draw (2,0.5) node[anchor=north west] {An initial T-mesh};
\draw (9,0.5) node[anchor=north west] {T-mesh after first modification};
\begin{scriptsize}
\draw [fill=black] (2.,4.5) circle (2.0pt);
\draw[color=black] (1.8,4.65) node {$v_6$};
\draw [fill=black] (5.,4.5) circle (2.0pt);
\draw[color=black] (5.2,4.65) node {$v_7$};
\draw [fill=black] (2.5,5.) circle (2.0pt);
\draw[color=black] (2.648326588350469,5.2) node {$v_4$};
\draw [fill=black] (2.5,4.) circle (2.0pt);
\draw[color=black] (2.65,4.2) node {$v_8$};
\draw [fill=black] (3.5,6.) circle (2.0pt);
\draw[color=black] (3.65,6.2) node {$v_{1}$};
\draw [fill=black] (3.5,3.) circle (2.0pt);
\draw[color=black] (3.7,3.2) node {$v_{12}$};
\draw [fill=black] (4.5,5.) circle (2.0pt);
\draw[color=black] (4.7,5.2) node {$v_5$};
\draw [fill=black] (4.5,2.) circle (2.0pt);
\draw[color=black] (4.7,2.2) node {$v_{16}$};
\draw [fill=black] (3.,3.5) circle (2.0pt);
\draw[color=black] (2.7,3.5) node {$v_{10}$};
\draw [fill=black] (6.,3.5) circle (2.0pt);
\draw[color=black] (6.3,3.5) node {$v_{11}$};
\draw [fill=black] (4.,2.5) circle (2.0pt);
\draw[color=black] (3.7,2.5) node {$v_{14}$};
\draw [fill=black] (5.,2.5) circle (2.0pt);
\draw[color=black] (5.3,2.5) node {$v_{15}$};
\draw [fill=black] (5.5,4.) circle (2.0pt);
\draw[color=black] (5.7,4.2) node {$v_9$};
\draw [fill=black] (5.5,3.) circle (2.0pt);
\draw[color=black] (5.7,3.2) node {$v_{13}$};
\draw [fill=black] (3.,5.5) circle (2.0pt);
\draw[color=black] (3.2,5.7) node {$v_2$};
\draw [fill=black] (4.,5.5) circle (2.0pt);
\draw[color=black] (4.2,5.7) node {$v_3$};
\draw [fill=ffqqqq] (10.,5.5) circle (2.0pt);
\draw[color=ffqqqq] (9.7,5.5) node {$w_3$};
\draw [fill=ffqqqq] (10.5,6.) circle (2.0pt);
\draw[color=ffqqqq] (10.7,6.2) node {$w_1$};
\draw [fill=ffqqqq] (10.,3.5) circle (2.0pt);
\draw[color=ffqqqq] (9.7,3.5) node {$w_5$};
\draw [fill=black] (11.,3.5) circle (2.0pt);
\draw[color=black] (11.2,3.7) node {$v_{10}$};
\draw [fill=ffqqqq] (10.5,3.) circle (2.0pt);
\draw[color=ffqqqq] (10.5,2.75) node {$w_6$};
\draw [fill=black] (10.,4.5) circle (2.0pt);
\draw[color=black] (9.7,4.5) node {$v_6$};
\draw [fill=black] (13.,4.5) circle (2.0pt);
\draw[color=black] (13.3,4.5) node {$v_7$};
\draw [fill=black] (11.5,6.) circle (2.0pt);
\draw[color=black] (11.7,6.2) node {$v_1$};
\draw [fill=black] (11.5,3.) circle (2.0pt);
\draw[color=black] (11.5,2.75) node {$v_{12}$};
\draw [fill=black] (14.,3.5) circle (2.0pt);
\draw[color=black] (14.3,3.5) node {$v_{11}$};
\draw [fill=black] (12.5,5) circle (2.0pt);
\draw[color=black] (12.7,5.2) node {$v_5$};
\draw [fill=black] (12.5,2) circle (2.0pt);
\draw[color=black] (12.5,1.7) node {$v_{16}$};
\draw [fill=qqqqff] (10.,6.) circle (2.5pt);
\draw[color=qqqqff] (9.7,6.3) node {$A$};
\draw [fill=qqqqff] (13.,6.) circle (2.5pt);
\draw[color=qqqqff] (13.3,6.3) node {$B$};
\draw [fill=qqqqff] (10.,3.) circle (2.5pt);
\draw[color=qqqqff] (9.7,2.7) node {$C$};
\draw [fill=qqqqff] (13.,3.) circle (2.5pt);
\draw[color=qqqqff] (12.8,3.2) node {$D$};
\draw [fill=black] (12.,2.5) circle (2.0pt);
\draw[color=black] (11.7,2.5) node {$v_{14}$};
\draw [fill=black] (13.,2.5) circle (2.0pt);
\draw[color=black] (13.3,2.5) node {$v_{15}$};
\draw [fill=black] (13.5,4.) circle (2.0pt);
\draw[color=black] (13.5,4.2) node {$v_9$};
\draw [fill=black] (13.5,3.) circle (2.0pt);
\draw[color=black] (13.7,3.2) node {$v_{13}$};
\draw [fill=ffqqqq] (13.,5.5) circle (2.0pt);
\draw[color=ffqqqq] (13.2,5.7) node {$w_4$};
\draw [fill=ffqqqq] (12.5,6.) circle (2.0pt);
\draw[color=ffqqqq] (12.7,6.2) node {$w_2$};
\draw [fill=black] (11,5.) circle (2.0pt);
\draw[color=black] (11.2,5.2) node {$E$};
\draw [fill=black] (11,2.) circle (2.0pt);
\draw[color=black] (11.2,2.2) node {$F$};
\draw [fill=black] (14,5.) circle (2.0pt);
\draw[color=black] (14.2,5.2) node {$H$};
\draw [fill=black] (14,2.) circle (2.0pt);
\draw[color=black] (14.2,2.2) node {$G$};
\end{scriptsize}
\end{tikzpicture}
\caption{\label{fig11}(Color online) T-mesh modification process}
\end{figure}
\begin{figure}[htbp]
    \centering
\definecolor{ttffcc}{rgb}{0.2,1.,0.8}
\definecolor{ffqqqq}{rgb}{1.,0.,0.}
\begin{tikzpicture}[line cap=round,line join=round,>=triangle 45,x=1.0cm,y=1.0cm]
\draw [line width=1.pt] (1.,1.)-- (7.,1.);
\draw [line width=1.pt] (1.,1.)-- (1.,7.);
\draw [line width=1.pt] (1.,7.)-- (7.,7.);
\draw [line width=1.pt] (7.,7.)-- (7.,1.);
\draw [line width=1.pt] (1.,6.)-- (7.,6.);
\draw [line width=1.pt] (1.,5.)-- (7.,5.);
\draw [line width=1.pt] (1.,4.)-- (7.,4.);
\draw [line width=1.pt] (1.,3.)-- (7.,3.);
\draw [line width=1.pt] (1.,2.)-- (7.,2.);
\draw [line width=1.pt] (2.,1.)-- (2.,7.);
\draw [line width=1.pt] (3.,7.)-- (3.,1.);
\draw [line width=1.pt] (4.,1.)-- (4.,7.);
\draw [line width=1.pt] (5.,7.)-- (5.,1.);
\draw [line width=1.pt] (6.,1.)-- (6.,7.);
\draw [line width=1.pt] (2.,4.5)-- (5.,4.5);
\draw [line width=1.pt] (3.,3.5)-- (6.,3.5);
\draw [line width=1.pt] (3.5,6.)-- (3.5,3.);
\draw [line width=1.pt] (4.5,5.)-- (4.5,2.);
\draw [line width=1.pt] (3.,5.5)-- (4.,5.5);
\draw [line width=1.pt] (2.5,5.)-- (2.5,4.);
\draw [line width=1.pt] (5.5,4.)-- (5.5,3.);
\draw [line width=1.pt] (4.,2.5)-- (5.,2.5);
\draw [line width=1.pt,dotted,color=ffqqqq] (2.5,6.)-- (2.5,5.);
\draw [line width=1.pt,dotted,color=ffqqqq] (2.,5.5)-- (3.,5.5);
\draw [line width=1.pt,dotted,color=ffqqqq] (2.5,4.)-- (2.5,3.);
\draw [line width=1.pt,dotted,color=ffqqqq] (3.,3.5)-- (2.,3.5);
\draw [line width=1.pt,dotted,color=ffqqqq] (4.,5.5)-- (5.,5.5);
\draw [line width=1.pt,dotted,color=ffqqqq] (4.5,5.)-- (4.5,6.);
\draw [line width=1.pt,dotted,color=ffqqqq] (3.5,3.)-- (3.5,2.);
\draw [line width=1.pt,dotted,color=ffqqqq] (4.,2.5)-- (3.,2.5);
\draw [line width=1.pt,dotted,color=ffqqqq] (5.,2.5)-- (6.,2.5);
\draw [line width=1.pt,dotted,color=ffqqqq] (5.5,4.)-- (5.5,5.);
\draw [line width=1.pt,dotted,color=ffqqqq] (5.,4.5)-- (6.,4.5);
\draw [line width=1.pt,dotted,color=ffqqqq] (5.5,3)-- (5.5,2.);
\draw [line width=2.4pt,color=ttffcc] (2.,6.)-- (2.,3.);
\draw [line width=2.4pt,color=ttffcc] (2.,6.)-- (5.,6.);
\draw [line width=2.4pt,color=ttffcc] (5.,6.)-- (5.,3.);
\draw [line width=2.4pt,color=ttffcc] (5.,3.)-- (2.,3.);
\draw [line width=2.4pt,color=ttffcc] (3.,5.)-- (3.,2.);
\draw [line width=2.4pt,color=ttffcc] (3.,2.)-- (6.,2.);
\draw [line width=2.4pt,color=ttffcc] (6.,2.)-- (6.,5.);
\draw [line width=2.4pt,color=ttffcc] (6.,5.)-- (3.,5.);
\end{tikzpicture}
    \caption{\label{fig12}(Color online) Modified T-mesh}
\end{figure}

\section{Hierarchical T-mesh and its CVR graph}
As an application of the results in previous sections, we prove a conjecture concerning the relationship between the dimension of \rev{$\bar S_d(\mathscr T)$} and that of \rev{$\bar S_{d-2}(\mathscr C)$}, where $\mathscr C$ is the CVR (cross-vertex relationship) graph of the hierarchical T-mesh $\mathscr T$. The CVR conjecture posits that constructing a set of basis functions for the spline space over the CVR graph, which is a lower-degree spline space with a readily constructible basis, enables the derivation of a basis for the spline space over the original hierarchical T-mesh\rev{, once a constructive isomorphism between the two spaces has been supplied}. Details on constructing basis functions for spline spaces over hierarchical T-meshes are provided in~\cite{deng2017bases,liu2023space,liu2025isogeometric}, though these constructions are limited to specific degrees, such as biquadratic and bicubic. In this paper, we address the general degree $d$.

\begin{definition}[See~\cite{jin2013}]
Given a T-mesh $\mathscr T$, we call a graph related to $\mathscr T$ the \textbf{CVR graph of $\mathscr T$} if it is constructed by retaining the cross vertices and the edges with two endpoints that are cross-vertices and removing the other vertices and the edges in $\mathscr T$.
\end{definition}

Figure~\ref{fig13} illustrates a T-mesh (shown on the left) and its corresponding CVR graph (shown on the right).
\begin{figure}[htbp]
    \centering
    \begin{tikzpicture}[line cap=round,line join=round,>=triangle 45,x=1.0cm,y=1.0cm]
\draw [line width=1.pt] (1.,1.)-- (7.,1.);
\draw [line width=1.pt] (1.,1.)-- (1.,7.);
\draw [line width=1.pt] (1.,7.)-- (7.,7.);
\draw [line width=1.pt] (7.,7.)-- (7.,1.);
\draw [line width=1.pt] (1.,6.)-- (7.,6.);
\draw [line width=1.pt] (1.,5.)-- (7.,5.);
\draw [line width=1.pt] (1.,4.)-- (7.,4.);
\draw [line width=1.pt] (1.,3.)-- (7.,3.);
\draw [line width=1.pt] (1.,2.)-- (7.,2.);
\draw [line width=1.pt] (2.,1.)-- (2.,7.);
\draw [line width=1.pt] (3.,7.)-- (3.,1.);
\draw [line width=1.pt] (4.,1.)-- (4.,7.);
\draw [line width=1.pt] (5.,7.)-- (5.,1.);
\draw [line width=1.pt] (6.,1.)-- (6.,7.);
\draw [line width=1.pt] (9.,6.)-- (9.,2.);
\draw [line width=1.pt] (9.,6.)-- (13.,6.);
\draw [line width=1.pt] (13.,6.)-- (13.,2.);
\draw [line width=1.pt] (13.,2.)-- (9.,2.);
\draw [line width=1.pt] (10.,6.)-- (10.,2.);
\draw [line width=1.pt] (11.,2.)-- (11.,6.);
\draw [line width=1.pt] (12.,6.)-- (12.,2.);
\draw [line width=1.pt] (13.,5.)-- (9.,5.);
\draw [line width=1.pt] (9.,4.)-- (13.,4.);
\draw [line width=1.pt] (9.,3.)-- (13.,3.);
\draw [line width=1.pt] (9.5,5.5)-- (11.5,5.5);
\draw [line width=1.pt] (9.5,5.5)-- (9.5,3.5);
\draw [line width=1.pt] (9.5,4.5)-- (12.5,4.5);
\draw [line width=1.pt] (12.5,4.5)-- (12.5,2.5);
\draw [line width=1.pt] (10.5,5.5)-- (10.5,2.5);
\draw [line width=1.pt] (10.5,2.5)-- (12.5,2.5);
\draw [line width=1.pt] (9.5,3.5)-- (12.5,3.5);
\draw [line width=1.pt] (11.5,5.5)-- (11.5,2.5);
\draw [line width=1.pt] (2.5,6.)-- (2.5,3.);
\draw [line width=1.pt] (2.,5.5)-- (5.,5.5);
\draw [line width=1.pt] (4.5,6.)-- (4.5,2.);
\draw [line width=1.pt] (2.,3.5)-- (6.,3.5);
\draw [line width=1.pt] (2.,4.5)-- (6.,4.5);
\draw [line width=1.pt] (3.5,6.)-- (3.5,2.);
\draw [line width=1.pt] (3.,2.5)-- (6.,2.5);
\draw [line width=1.pt] (5.5,5.)-- (5.5,2.);
\draw (3.5,0.6) node[anchor=north west] {$\mathscr{T}$};
\draw (10.5,0.6) node[anchor=north west] {$\mathscr{C}$};
\end{tikzpicture}
    \caption{\label{fig13}A T-mesh $\mathscr T$ and its CVR graph $\mathscr C$}
\end{figure}
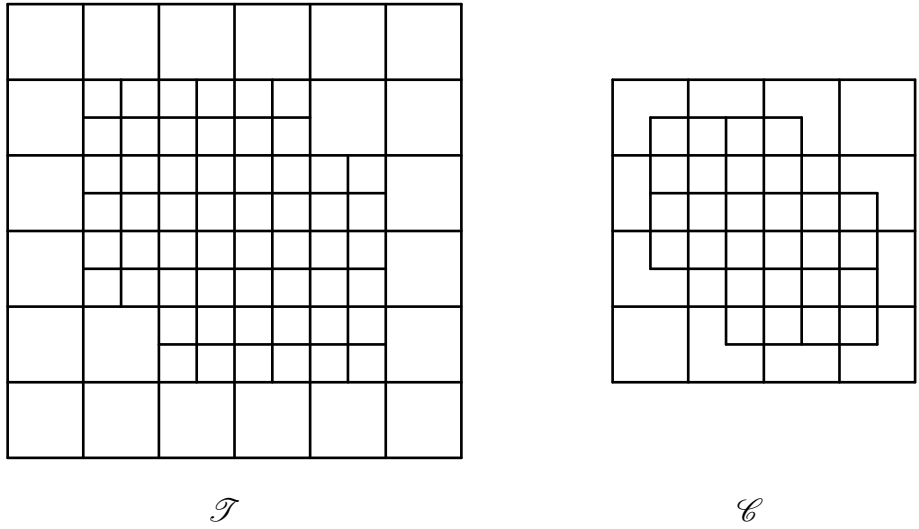

In~\cite{jin2013}, Deng et al.\ proposed a conjecture regarding the relationship between the dimension of the spline space over a hierarchical T-mesh and the dimension of the spline space over its CVR graph. In the following, we present a version of this conjecture specifically for spline spaces with the highest order of smoothness.

\begin{conjecture}\label{conj}
Let $\mathscr T$ be a hierarchical T-mesh with the CVR graph being $\mathscr C$. When $d\ge2$, one has $\dim\bar S_d(\mathscr T)=\dim\bar S_{d-2}(\mathscr C)$.
\end{conjecture}

When a hierarchical T-mesh is subdivided using the cross subdivision mode and all the subdivided cells are interior cells, the resulting CVR graph can be treated as a regular T-mesh with some $L$-nodes (interior vertices connected by two edges). An example is shown in Figure~\ref{fig13}. It is important to note that all previous results based on the smoothing cofactor method can be extended to T-meshes with $L$-nodes. Therefore, for the remainder of this paper, we assume that all the subdivided cells are interior cells. Finally, a general method for processing the case of subdividing boundary cells will be presented. Before proving the conjecture, we first introduce two lemmas.

\begin{lemma}\label{lem4.1}
Given a tensor product mesh $\mathscr T_0$, let $\mathscr T$ be the T-mesh obtained by subdividing a collection of $(d-1)\times(d-1)$ submeshes of $\mathscr T_0$ under cross subdivision, where the submeshes can partially overlap. Let $\mathscr C$ be the CVR graph of $\mathscr T$. Then
\[
\dim\textsf{CVS}_1[L(\mathscr T)]=\dim\textsf{CVS}_2[L(\mathscr C)],
\]
where $\textsf{CVS}_1$ and $\textsf{CVS}_2$ are the conformality vector space corresponding to the spline spaces $S_d(\mathscr T)$ and $S_{d-2}(\mathscr C)$, respectively, and $L(\mathscr C)$ is the set of all the T $l$-edges in $\mathscr C$.
\end{lemma}

\begin{proof}
Let $\{L_1,L_2\}$ be a bipartite partition of $L(\mathscr T)$, where $N(l)=d-1$ for any $l\in L_1$ and $N(l)>d-1$ for any $l\in L_2$. Here, $N(l)$ represents the number of cells crossed by the edge $l$. Based on Theorem~\ref{thm4.1}, we only prove the conclusion for the following two cases.
\begin{enumerate}
\item[(1)] The case where $L_2\ne\varnothing$ and there is no isolated $(d-1)\times(d-1)$ subdivision.

By Theorem~\ref{thm4.1}, in this case, all T $l$-edges in $L(\mathscr T)$ have a reasonable order denoted as $l_1\succ\cdots\succ l_t$. Then
\[
\dim\textsf{CVS}_1[L(\mathscr T)]=v-(d+1)t,
\]
where $v$ is the number of vertices in $L(\mathscr T)$ and $t$ is the number of T $l$-edges in $L(\mathscr T)$.

On the other hand, suppose that $\bar l_i$ is the T $l$-edge in the CVR graph $\mathscr C$ corresponding to $l_i$ in $\mathscr T$, for $i=1,2,\ldots,t$. Since $\bar l_i$ is obtained by excluding the two endpoints of $l_i$, the order $\bar l_1\succ\cdots\succ\bar l_t$ is also a reasonable order when considering the spline space $S_{d-2}(\mathscr C)$. Thus, $\dim\textsf{CVS}_2[L(\rev{\mathscr C})]=v_{\mathscr C}-(d-1)t_{\mathscr C}$, where $v_{\mathscr C}$ is the number of vertices in $L(\mathscr C)$ and $t_{\mathscr C}$ is the number of T $l$-edges in $L(\mathscr C)$.

As $v_{\mathscr C}=v-2t$ and $t_{\mathscr C}=t$, it follows that
\[
\dim\textsf{CVS}_2[L(\rev{\mathscr C})]=v_{\mathscr C}-(d-1)t_{\mathscr C}=v-(d+1)t=\dim\textsf{CVS}_1[L(\mathscr T)].
\]

\item[(2)] The case where $L_2=\varnothing$.

By Theorem~\ref{thm4.1}, in this case, $\mathscr T_0$ is subdivided with a collection of isolated $(d-1)\times(d-1)$ submeshes. One can consider the subdivision of only one $(d-1)\times(d-1)$ submesh. According to Theorem~\ref{thm3.1}, the dimension of the conformality vector space corresponding to $S_d(\mathscr T)$ is given by
\[
\dim\textsf{CVS}_1[L(\mathscr T)]=[2d-1-(d+1)][2d-1-(d+1)]=(d-2)^2.
\]
On the other hand, the set of all the T $l$-edges in $L(\mathscr C)$ is also a tensor product T-connected component with $2d-3$ vertices in each $l$-edge. Hence, by Theorem~\ref{thm3.1}, the dimension of the conformality vector space corresponding to $S_{d-2}(\mathscr C)$ is
\[
\dim\textsf{CVS}_2[L(\rev{\mathscr C})]=[2d-3-(d-1)][2d-3-(d-1)]=(d-2)^2.
\]
As a result, we have $\dim\textsf{CVS}_1[L(\mathscr T)]=\dim\textsf{CVS}_2[L(\rev{\mathscr C})]$.\hfill$\Box$
\end{enumerate}
\end{proof}

For a hierarchical T-mesh, the CVR graph can be derived on a level-by-level basis, resulting in a hierarchical structure. In the following, we prove that the conformality vector space of the T-connected component of the CVR graph can be calculated in a recursive manner using a similar method as in the proof of Theorem~\ref{thm3.2}.

\begin{lemma}\label{lem4.3}
Given a hierarchical T-mesh $\mathscr T$ under the cross subdivision mode, let $\mathscr C$ be the corresponding CVR graph of $\mathscr T$. In each level, the T-mesh is subdivided by a collection of $(d-1)\times(d-1)$ tensor product subdivisions, and these subdivision regions may partially overlap. Then the dimension of the conformality vector space of the T-connected component of the CVR graph can be calculated in a recursive manner as
\begin{equation}\label{CVR}
\dim\textsf{CVS}_2[L(\mathscr C)]=\sum_{i=1}^{\lev(\mathscr C)}\dim\textsf{CVS}_2[C_i],
\end{equation}
where $C_i$ is the set of all the T $l$-edges and vertices in level $i$.
\end{lemma}

\begin{proof}
We first prove the case where $\lev(\mathscr C)=2$. Consider $C_1$ as a $(d-1)\times(d-1)$ subdivision. Otherwise, all $l$-edges in $C_1$ are in a reasonable order, and thus the equation~\eqref{CVR} holds immediately.

However, in this case, it is not clear which cells in level~1 are subdivided by $C_2$. To address this issue, \rev{we define $C'_2$ as the level-2 part of the CVR graph obtained by subdividing all $4(d-1)^2$ level-1 cells in the corresponding original $(d-1)\times(d-1)$ parent region.} Note that $C_2$ is a subset of $C'_2$, and these two subdivisions satisfy the conditions in Lemma~\ref{lem2.1}. If we can prove that $\dim\textsf{CVS}_2[L(\mathscr C')]=\dim\textsf{CVS}_2[C_1]+\dim\textsf{CVS}_2[C'_2]$, where $L(\mathscr C')$ includes $C_1$ and $C'_2$, then the equation~\eqref{CVR} follows by Lemma~\ref{lem2.1}.

Since $L(\mathscr C')$ forms a $(4d-5)\times(4d-5)$ tensor-product mesh $\mathscr T'$, it follows that
\begin{align*}
\dim\textsf{CVS}_2[L(\mathscr C')]&=\dim\bar S_{d-2}(\mathscr T')\\
&=[4d-5-(d-2+1)][4d-5-(d-2+1)]=(3d-4)^2.
\end{align*}
Using Theorem~\ref{thm3.1}, we have
\[
\dim\textsf{CVS}_2[C_1]=[(2d-3)-(d-1)][(2d-3)-(d-1)]=(d-2)^2.
\]
For $C'_2$, since all the $l$-edges have a reasonable order, it has $4(d-1)(3d-4)$ vertices and $4(d-1)$ $l$-edges. Thus,
\[
\dim\textsf{CVS}_2[C'_2]=4(d-1)(3d-4)-4(d-1)^2=4(d-1)(2d-3).
\]
Therefore,
\begin{align*}
\dim\textsf{CVS}_2[C_1]+\dim\textsf{CVS}_2[C'_2]&=(d-2)^2+4(d-1)(2d-3)\\
&=9d^2-24d+16=(3d-4)^2.
\end{align*}
Hence, $\dim\textsf{CVS}_2[L(\mathscr C')]=\dim\textsf{CVS}_2[C_1]+\dim\textsf{CVS}_2[C'_2]$. The case of $\lev(\mathscr C)=2$ is then proved. For the general case, it follows as a recursive process. \hfill$\Box$
\end{proof}

Similar to the previous section, Lemmas~\ref{lem4.1} and~\rev{\ref{lem4.3}} can be generalized to the case of the spline space with homogeneous boundary conditions. Now we are able to prove Conjecture~\ref{conj} over hierarchical T-meshes under $(d-1)\times(d-1)$ tensor product subdivisions.

\begin{theorem}\label{thm5.5}
Let $\mathscr T$ be a hierarchical T-mesh and $\mathscr C$ be the CVR of $\mathscr T$. Suppose that the T-mesh $\mathscr T$ is obtained by subdividing a collection of $(d-1)\times(d-1)$ tensor product submeshes in each level, where the submeshes may overlap. Then $\dim\bar S_d(\mathscr T)=\dim\bar S_{d-2}(\mathscr C)$.
\end{theorem}

\begin{proof}
By Lemma~\ref{lem4.3} in the case of $\bar S_{\rev{d-2}}(\mathscr C)$, one has
\[
\dim\bar S_{d-2}(\mathscr C)=\dim\bar S_{d-2}(\mathscr C_0)+\sum_{i=1}^{\lev(\mathscr C)}\dim\textsf{CVS}_2[\bar C_i],
\]
where $\bar C_i$ is the set of all T $l$-edges of the subdivisions of $\mathscr C$ in level $i$.

Suppose that $\mathscr T_0$ is \rev{a tensor-product mesh with $m$ horizontal and $n$ vertical grid lines}, then $\mathscr C_0$ is \rev{a tensor-product mesh with $m-2$ horizontal and $n-2$ vertical grid lines}, and
\begin{align*}
\dim\bar S_{d-2}(\mathscr C_0)&=\rev{[(m-2)-(d-2+1)]_+[(n-2)-(d-2+1)]_+}\\
&=\rev{(m-d-1)_+(n-d-1)_+}.
\end{align*}
\rev{Here $(u)_+=\max\{u,0\}$.} On the other hand,
\[
\dim\bar S_d(\mathscr T_0)=\rev{[m-(d+1)]_+[n-(d+1)]_+=(m-d-1)_+(n-d-1)_+},
\]
then $\dim\bar S_d(\mathscr T_0)=\dim\bar S_{d-2}(\mathscr C_0)$.

By Lemma~\ref{lem4.1} in the case of $\bar S_{\rev{d-2}}(\mathscr C)$ and $\bar S_d(\mathscr T)$, $\dim\textsf{CVS}_1[\bar T_i]=\dim\textsf{CVS}_2[\bar C_i]$, where $\bar T_i$ is the set of all T $l$-edges of the subdivisions of $\mathscr T$ in level $i$.
Thus,
\begin{align*}
\dim\bar S_{d-2}(\mathscr C)&=\dim\bar S_{d-2}(\mathscr C_0)+\sum_{i=1}^{\lev(\mathscr C)}\dim\textsf{CVS}_2[\bar C_i]\\
&=\dim\bar S_d(\mathscr T_0)+\sum_{i=1}^{\lev(\mathscr T)}\dim\textsf{CVS}_{\rev{1}}[\bar T_i]\\
&=\dim\bar S_d(\mathscr T).
\end{align*}
This completes the proof. \hfill$\Box$
\end{proof}

As a final remark, if there are some subdivided cells which are boundary cells, the approach outlined in Lemma~\ref{lem2.1} can be implemented by subdividing all the boundary cells of $\mathscr T$. This will then convert the CVR graph into a regular T-mesh. Thus, the same result can be obtained in this case.

\section{Acknowledgement}
This work was supported by the Key Project of National Natural Science Foundation of China (Grant No.\ 12494550) and National Natural Science Foundation of China (Grant No.\ 12371383).

\bigskip

\section{Conclusion and Future Work}\label{sec:conclusion}
Our study investigates the dimension of the spline space $S_d(\mathscr T)$, which has the highest order of smoothness over a hierarchical T-mesh $\mathscr T$ using the smoothing cofactor conformality method. To begin, we introduce the concept of a tensor product T-connected component and derive a dimension formula for the conformality vector space through the smoothing cofactor method. Our proof shows that the dimension of the conformality vector space over a T-connected component of a hierarchical T-mesh under tensor product subdivision can be computed recursively. Thus, the results in this paper generalize previous results, as previous studies' subdivision modes could be considered as a special case of our tensor product subdivision. Based on the above results, we obtain a dimensional formula for the spline space $S_d(\mathscr T)$ over a hierarchical T-mesh $\mathscr T$ under $(d-1)\times(d-1)$ tensor product subdivisions. Finally, we confirm a conjecture regarding the relationship between the dimension of the spline space over a hierarchical T-mesh $\mathscr T$ and that over its CVR graph.

There are several open questions for future research. For example, if the tensor product subdivision does not meet the condition in Theorem~\ref{thm3.2}, the dimension of the conformality vector space of the T-connected component \rev{is not guaranteed by that theorem to be calculable level by level, and the cases not covered by the theorem deserve further study.} Another point of investigation is the dimensions of the highest-order smoothness spline space with bi-degree $(m,n)$, where $m\ne n$.

\bibliographystyle{IEEEtran}
\bibliography{T3_v3}

\end{document}